\pdfoutput=1

\documentclass[12pt]{amsart}
\usepackage[margin=1in]{geometry} 

\usepackage{amsmath,amssymb,amsthm,bm,colonequals,graphicx,mathrsfs,mathtools,microtype,multicol,stmaryrd}
\usepackage[shortlabels]{enumitem}
\usepackage[colorinlistoftodos]{todonotes}

\numberwithin{equation}{section}

\renewcommand{\Re}{\operatorname{Re}}
\renewcommand{\Im}{\operatorname{Im}}

\theoremstyle{plain}

\newtheorem{theorem}{Theorem}[section] 
\newtheorem*{theorem*}{Theorem}

\newtheorem{lemma}[theorem]{Lemma} 

\newtheorem{proposition}[theorem]{Proposition} 

\newtheorem{proposition-definition}[theorem]{Proposition-Definition} 

\newtheorem{corollary}[theorem]{Corollary} 

\theoremstyle{definition}

\newtheorem{definition}[theorem]{Definition} 

\theoremstyle{remark}

\newtheorem{remark}[theorem]{Remark} 

\newcommand{\Aff}{\mathbb{A}} 

\newcommand{\CC}{\mathbb{C}}

\newcommand{\FF}{\mathbb{F}}

\newcommand{\PP}{\mathbb{P}}
\newcommand{\QQ}{\mathbb{Q}}
\newcommand{\RR}{\mathbb{R}}

\newcommand{\ZZ}{\mathbb{Z}}

\newcommand{\tr}{\operatorname{tr}}
\newcommand{\rank}{\operatorname{rank}}

\newcommand{\abs}[1]{\lvert #1 \rvert}
\newcommand{\card}[1]{\lvert #1 \rvert}

\newcommand{\floor}[1]{\lfloor #1 \rfloor}

\newcommand{\eps}{\epsilon}

\newcommand{\Hom}{\operatorname{Hom}}

\newcommand{\disc}{\operatorname{disc}}

\newcommand{\Gal}{\operatorname{Gal}}

\newcommand{\map}{\operatorname}
\newcommand{\mscr}{\mathscr}
\newcommand{\mcal}{\mathcal}

\newcommand{\ol}{\overline}

\newcommand{\defeq}{\colonequals}

\newcommand{\maps}{\colon}

\newcommand{\belongs}{\subseteq}

\newcommand{\set}[1]{\{#1\}}

\def\OK{\mathcal{O}}
\def\Omon{\OK^+}

\newcommand{\VV}{\mathbb{V}}
\newcommand{\Schur}{\mathbb{S}}
\newcommand{\GL}{\operatorname{GL}}

\newcommand{\Orth}{\operatorname{O}}

\newcommand{\Sp}{\operatorname{Sp}}

\newcommand{\ptn}{\mathsf{P}}
\newcommand{\Fr}{\map{Frob}}

\newcommand{\Crank}{C_0}
\newcommand{\Cbetti}{C_1}
\newcommand{\Cslope}{C_2}
\newcommand{\Cstart}{C_3}
\newcommand{\Cfam}{C_4}
\newcommand{\Cavg}{C_5}
\newcommand{\Cstabone}{C_6}
\newcommand{\Cstabtwo}{C_7}
\newcommand{\Cstabthr}{C_8}
\newcommand{\Cstabfou}{C_9}

\usepackage[pagebackref=true]{hyperref}
\usepackage[alphabetic,initials]{amsrefs}

\begin{document}

\title{Notes on zeta ratio stabilization}

\date{} 
\author{Victor Y. Wang}
\address{IST Austria, Am Campus 1, 3400 Klosterneuburg, Austria}
\email{vywang@alum.mit.edu}


\begin{abstract}
This semi-expository note clarifies the extent to which recent ideas in homological stability can resolve the Ratios Conjecture over $\mathbb{F}_q(t)$.
For large fixed $q$,
a uniform power saving at distance $\ge q^{-\delta}$ from the critical line is possible.
This implies cancellation-beyond-GRH
in arbitrarily large ranges of moduli relative to the family of $L$-functions.
It has applications to the statistics of low-lying zeros.
\end{abstract}

\maketitle


\section{Introduction}
\label{SEC:intro}

The conjectures of \cites{katz1999random,conrey2005integral,conrey2008autocorrelation} have inspired much work on
the moment and ratio statistics of $L$-functions.
Roughly speaking, the literature includes analytic methods over global fields
(e.g.~Poisson summation, trace formulas,
and multiple Dirichlet series,
as well as methods building on probabilistic insights of \cites{soundararajan2009moments,harper2013sharp}),
and geometric methods over global function fields (based on the Grothendieck--Lefschetz trace formula as in \cite{kowalski2022binary}).
The geometric approaches include, but are not limited to, the following (for large fixed $q$):
\begin{itemize}
    \item relate behavior in a small family to a larger family \cite{sawin2022square};

    \item use a range of vanishing or stable (co)homology \cites{sawin2020representation,bergstrom2023hyperelliptic}.
\end{itemize}
Homological stability and vanishing are often closely related,
but have different conceptual advantages.
The stability perspective will be most convenient for us below.

The recent uniform stability result \cite{MPPRW}*{Theorem~1.4}
completes the program of \cite{bergstrom2023hyperelliptic}, resolving the $K$th Moments Conjecture for quadratic Dirichlet $L$-functions over $\FF_q(t)$ for odd $q\gg_K 1$.
The goal of the present note is to expand the scope of the methods to a similarly \emph{$q$-restricted} form of the Ratios Conjecture, unconditionally for quadratic Dirichlet $L$-functions and axiomatically for more general families.

When discussing the Ratios Conjecture,
there are two key aspects to focus on:
\begin{enumerate}
    \item the distance between $s$ and the critical line $\frac12+i\RR$;
    \item the quality of the error term in the asymptotic formula.
\end{enumerate}
Ideally, one would like to simultaneously resolve both aspects.
Most existing work focuses on (1) at the expense of (2);
see e.g.~\cites{bui2021ratios,florea2021negative,bui2023negative}.
But for (homological) geometric methods, (2) at the expense of (1) is more accessible, as will become clear below.

Let $q$ be a prime power.
Let $\OK = \FF_q[t]$.
Let $\Omon$ be the set of monic $f\in \OK$.
Let $\mscr{P}_n$ be the set of square-free $f\in \Omon$ of degree $n\ge 1$.
It is well known \cite{randal2019homology}*{\S3.1} that
$\card{\mscr{P}_n} = q^n - q^{n-1}$.
If $q$ is odd, then for each $f\in \mscr{P}_n$,
let $\chi_f(r) = (\frac{r}{f})$ be the Jacobi symbol over $\FF_q[t]$,
let $L(s,\chi_f) \defeq \sum_{r\in \Omon} \chi_f(r) \abs{r}^{-s}$ where $\abs{r} \defeq q^{\deg{r}}$,
and let
\begin{equation*}
\begin{split}
L_\infty(s,\chi_f) &\defeq (1-q^{-s})^{-(1+(-1)^n)/2}, \\
\Lambda(s,\chi_f) &\defeq L(s,\chi_f) L_\infty(s,\chi_f).
\end{split}
\end{equation*}
It is known that $L(s,\chi_f)\in \QQ(q^{-s})$ and $\Lambda(s,\chi_f) \in \ZZ[q^{-s}]$ \cite{bergstrom2023hyperelliptic}*{\S11.1.10}.

All ensuing results in this section build heavily on \cites{bergstrom2023hyperelliptic,MPPRW}.
We concentrate on odd values of $n$, but state the results in a way that might also hold for even $n$.

\begin{theorem}
\label{THM:apply-stability-1}

Fix integers $K,Q\ge 0$.
Let $\delta = \max(576, 2016(K+Q))^{-1}$ as in \eqref{EQN:quadratic-specialized-delta-value},
and let $\omega = \frac{1}{84}$.
Take an odd $q \ge 2^{12} 2^{1/\delta}$.
Let $\tfrac12 - \delta < \Re(s_1),\dots,\Re(s_K) < \tfrac12 + \delta$
and $\Re(s_{K+1}),\dots,\Re(s_{K+Q}) > \tfrac12 + q^{-\delta}$.
Assume $2\nmid n\ge 2Q+1$.
Then
\begin{equation}
\label{INEQ:main-quadratic-goal}
\biggl\lvert{\sum_{d\in \mscr{P}_n}
\frac{L(s_1,\chi_d) \cdots L(s_K,\chi_d)}
{L(s_{K+1},\chi_d) \cdots L(s_{K+Q},\chi_d)}
- \mathsf{RR}_L(\bm{s};n)}\biggr\rvert
\le 5 q^{O_{K,Q}(1)} q^{(1-\omega)n},
\end{equation}
where $\mathsf{RR}_L(\bm{s};n)$ is the main term in
the Ratios Recipe of \cites{conrey2005integral,conrey2008autocorrelation,andrade2014conjectures} over $\mscr{P}_n$
when there are $K$ factors of $L$ in the numerator
and $Q$ factors of $L$ in the denominator.
\end{theorem}


The following result illustrates the level of cancellation captured in Theorem~\ref{THM:apply-stability-1}.

\begin{corollary}
\label{COR:integrate-stability-1}

Fix a compact interval $I\belongs \RR$.
Let $\mu_{\FF_q[t]}(r)$ be the M\"{o}bius function over $\FF_q[t]$.
If $2\nmid n\ge 3$ and $R\in \ZZ$ with $R/n\in I$,
and $q$ is odd and sufficiently large in terms of $I$,
then
\begin{equation}
\label{EQN:illustrate-cancellation-negative-first-moment}
\frac{1}{\card{\mscr{P}_n}} \sum_{d\in \mscr{P}_n}
\frac{1}{q^{R/2}} \sum_{r\in \Omon:\, \abs{r}=q^R} \mu_{\FF_q[t]}(r) \chi_d(r)
= \bm{1}_{R=0}
+ O(q^{O(1)} q^{-\omega n/2}).
\end{equation}
\end{corollary}




A key feature of Theorem~\ref{THM:apply-stability-1} is the uniform power-saving error term over arbitrary small fixed values of $\Re(s)-\frac12>0$, which allows for the natural application Corollary~\ref{COR:integrate-stability-1} for arbitrarily large intervals $I$.
A weaker power saving than $q^{-\omega n}$ in Theorem~\ref{THM:apply-stability-1}, e.g.~$q^{-f(\Re(s) - 1/2) n}$ where $\sup_{\beta>0}(f(\beta)/\beta) = \infty$, would also suffice, up to adjusting the power saving in \eqref{EQN:illustrate-cancellation-negative-first-moment}.
But it seems that in prior work, the ``slope'' $\sup_{\beta>0}(f(\beta)/\beta)$ has always been finite.

On the other hand, Theorem~\ref{THM:apply-stability-1} for $(K,Q)=(0,1)$ is worse than \cite{florea2021negative}*{Theorem~1.3} when $\Re(s)-\frac12$ decays in $n$.
In particular, it remains an open problem to prove even a little-$o$ asymptotic for $\sum \frac1L$ when $\Re(s)-\frac12 \asymp \frac1n$.
This range is of interest in its own right, and might represent a transition range \cite{bui2023negative}*{Conjecture~1.1 (Gonek)}; but also, in applications, being able to take $\Re(s)-\frac12 \asymp \frac1n$ would improve uniformity in $I$.
Thus, the full Ratios Conjecture for quadratic Dirichlet $L$-functions remains open, even though for certain applications like Corollary~\ref{COR:integrate-stability-1} (and similarly for Theorem~\ref{THM:one-level-density} below) it is resolved.

Next, we analyze zero statistics with arbitrarily large fixed Fourier support $I$ (for $q$ large enough in terms of $I$).
For simplicity, we concentrate on the One-Level Density Conjecture \cite{katz1999zeroes}*{(42)}.
The following result goes beyond the $I\belongs (-2,2)$ range of \cite{rudnick2008traces}*{Corollary~3}, provided $q\gg 1$.
As explained in \cite{andrade2014conjectures}*{\S7}, one could remove the requirement $q\gg_I 1$ under a sufficiently strong form of the Ratios Conjecture for $(K,Q)=(1,1)$.

\begin{theorem}
\label{THM:one-level-density}
Fix a compact interval $I\belongs \RR$.
Let $g\maps \RR\to \CC$ be a smooth even function supported on $I$.
Let $f(x) \defeq \int_{\xi\in \RR} g(\xi) e(x \xi)\, d\xi$ be the inverse Fourier transform of $g$.
If $q$ is odd and sufficiently large in terms of $I$,
then as $n\to \infty$ with $2\nmid n$,
\begin{equation}
\label{1-level-density-convergence}
\frac{1}{\card{\mscr{P}_n}} \sum_{d\in \mscr{P}_n}\,
\sum_{\substack{\gamma_d\in \RR,\textnormal{ with multiplicity}: \\ L(\frac12 + i\gamma_d, \chi_d) = 0}}\,
f{\left(\frac{\gamma_d \log(q^{2\,\floor{(n-1)/2}})}{2\pi}\right)}
\to \int_\RR \left(1 - \frac{\sin(2\pi x)}{2\pi x}\right) f(x)\, dx.
\end{equation}
\end{theorem}

For proofs, see \S\ref{SEC:quadratic-analysis}.
But we first give in \S\ref{SEC:general-framework} a general axiomatization of Theorem~\ref{THM:apply-stability-1}.
One could improve $\omega$ with more work,
even without improving the slope $\frac1{12}$ in \cite{MPPRW}*{Theorem~1.4},
but in this note we prefer
clean, robust, general bounds.

An important role is played by the \emph{exponential} Betti bounds of \cite{bergstrom2023hyperelliptic}*{Lemma~11.3.13}.
Diaconu has informed us that polynomial, or partly polynomial, bounds may be plausible.
It would be very interesting to investigate this, and its implications, carefully.
If we knew, for instance, that
for local systems of the form $\bigwedge^{i_1} \otimes \dots \otimes \bigwedge^{i_K} \otimes \map{Sym}^{j_1} \otimes \dots \otimes \map{Sym}^{j_Q}$,
the Betti numbers associated to $\mscr{P}_n$
were $\le (2+n+i_1+\dots+j_Q)^{O_{K,Q}(1)}$,
then the $q$-restrictions above might be removable.
A \emph{sub-exponential} bound like $(2+n+i_1+\dots+j_Q)^{o(n/\log{n})}$ or $e^{o(2+n+i_1+\dots+j_Q)}$ might also suffice.
But if the Betti growth really is exponential, then one might need cancellation in unstable ranges to remove the $q$-restrictions.
Alternatively, it could be very interesting to combine,
or at least understand the relations between,
the somewhat complementary approaches exemplified in
\cites{bergstrom2023hyperelliptic,MPPRW,florea2021negative,sawin2020representation,sawin2022square}.

\section{General framework}
\label{SEC:general-framework}

We need Schur functors on finite-dimensional vector spaces (or local systems thereof) $V$ over a field of characteristic zero.
Let $\ptn$ be the set of all partitions $\lambda=(\lambda_1\ge\lambda_2\ge\cdots\ge 0)$.
Given $\lambda\in \ptn$,
define its \emph{size} $\abs{\lambda}\defeq \sum_{i\ge 0} \lambda_i$,
\emph{length} $l(\lambda)\defeq \#\set{i\ge 0: \lambda_i\ne 0}$,
and let
$$\Schur^{\GL}_\lambda(V)
\defeq \Schur_\lambda(V)
\belongs V^{\otimes \abs{\lambda}}$$
be the usual Schur functor associated to $\lambda$; see \cite{fulton2013representation}*{\S6.1}.
For $V$ equipped with a perfect pairing of type $G\in \set{\Sp,\Orth}$
(skew-symmetric if $G=\Sp$, and symmetric if $G=\Orth$), let
$$\Schur^G_\lambda(V)
\defeq {\textstyle \bigcap_{\binom{\abs{\lambda}}{2}}}
\ker(\Schur_\lambda(V) \belongs V^{\otimes \abs{\lambda}}\to V^{\otimes (\abs{\lambda}-2)})
\belongs \Schur_\lambda(V),$$
as in \cite{fulton2013representation}*{\S17.3, (17.10)} for $G=\Sp$ and \cite{fulton2013representation}*{\S19.5, (19.18)} for $G=\Orth$.

For each integer $n\ge 1$, let $\mcal{P}_n$ be
a smooth, geometrically connected variety over $\FF_q$ of dimension $d(n)\in \ZZ$.
For technical convenience, we assume
(what is mild in practice)
\begin{equation}
\label{INEQ:dyadic-dimension-bound-assumption}
1\le d(n) \le d(n+1) \le 2d(n),
\qquad\textnormal{and}\qquad
\lim_{n\to \infty} d(n) = \infty.
\end{equation}
Suppose we have a family of
entire,\footnote{One could probably handle poles on the line $\Re(s) = 1$ with more work.
Many families have no poles.}
completed,\footnote{With more work, one could probably handle an additional twist by $L_\infty(s,f)^{-1}$.
In many families, the factor $L_\infty(s,f)$ does not vary with $f$, so working with $L$ and $\Lambda$ are equivalent.}
analytically normalized $L$-functions
\begin{equation*}
\Lambda(s,f) = \sum_{N\ge 0} A_f(N) q^{(1/2-s)N} \in \CC[q^{-s}]
\qquad (\textnormal{with }A_f(0) = 1)
\end{equation*}
over $\FF_q(t)$ indexed by $f\in \mcal{P}_n(\FF_q)$.
Let $c(f) \defeq \deg_{q^{-s}}(\Lambda(s,f)) \in \ZZ_{\ge 0}$.
Each $\Lambda(s,f)$
is $\frac{2\pi i}{\log{q}}$-periodic,
satisfies GRH (with critical line $\Re(s) = \frac12$),
and has a functional equation
\begin{equation}
\label{EQN:FE}
\Lambda(s,f)
= w(f) q^{c(f)(1/2-s)} \ol{\Lambda}(1-s,f),
\end{equation}
where $\abs{w(f)} = 1$.
Therefore, we have an \emph{approximate functional equation}
\begin{equation}
\label{EQN:approx-FE}
\Lambda(s,f) = \sum_{0\le N\le c(f)/2} A_f(N) q^{(1/2-s)N}
+ w(f) q^{c(f)(1/2-s)} \sum_{0\le M<c(f)/2} \ol{A}_f(M) q^{(s-1/2)M};
\end{equation}
cf.~\cite{andrade2014conjectures}*{(4.3)}.
Write
\begin{equation*}
\Lambda(s,f)^{-1} = \sum_{N\ge 0} B_f(N) q^{(1/2-s)N}.
\end{equation*}

Let $G\in \set{\GL,\Sp,\Orth}$.
Let $K,Q\ge 0$ be integers representing the number of factors in the numerator and denominator, respectively, of the Ratios Conjecture.
Let $\Crank,\Cbetti,\Cbetti',\Cslope,\Cstart\ge 1$ be constants independent of $n$.
Let $\Fr_q$ be the geometric Frobenius element of $\Gal(\ol{\FF}_q/\FF_q)$.
We now define a useful set of axioms in some generality.
We will avoid dual Schur functors $\Schur_\lambda(V^\ast)$, which would be needed for a full analysis of the case $G=\GL$.
\begin{enumerate}
    \item \emph{Axiom~L}(isse):
    For some prime $\ell\nmid q$,
    there exists a local system $V_n$ of $\ol{\QQ}_\ell$-vector spaces on $\mcal{P}_n$
    of some rank\footnote{This is the important case for us.
    The methods would also
    allow $c[n] \asymp d(n)^\theta$ for any fixed $\theta\in (0,1]$.}
    \begin{equation}
    \label{INEQ:linear-rank-growth}
    c[n] \in [\Crank^{-1} d(n) - 2, \Crank d(n)],
    \end{equation}
    pure of weight $0$,
    equipped with a perfect pairing $$V_n \times V_n \to \ol{\QQ}_\ell$$ of type $G$ if $G\in \set{\Sp,\Orth}$,\footnote{Equivalently, given a geometric point $x\in \mcal{P}_n$, we have a monodromy representation $\pi_1(\mcal{P}_n,x) \to G((V_n)_x)$.}
    such that
    \begin{equation}
    \label{EQN:completed-L-function-as-a-char-poly}
    \Lambda(s,f) = \det(1-q^{1/2-s}\Fr_q, (V_n)_f)
    \end{equation}
    for all $f\in \mcal{P}_n(\FF_q)$.
    In particular, $c(f) = c[n]$
    for all $f\in \mcal{P}_n(\FF_q)$.

    
    \item \emph{Axiom~B}(etti):
    For every local system $\mcal{L} = \Schur^G_\lambda(V_n) \otimes \Schur^G_\mu(V_n)$ with $\lambda,\mu\in \ptn$, $\lambda_1 \le K$, and $l(\mu) \le Q$, we have
    \begin{equation*}
        \sum_{k\ge 0} \dim{H_k(\mcal{P}_n \otimes \ol{\FF}_q, \mcal{L})}
        \le \Cbetti^{d(n)} (\Cbetti')^{\abs{\mu}} \rank{\mcal{L}}.
    \end{equation*}
    Here $H_k(X,\mcal{L})\defeq H^{2\dim(X)-k}_c(X,\mcal{L}(\dim{X}))$ for notational convenience.
    (By Poincar\'{e} duality, this is a reasonable definition of $H_k$, because our $X$ is smooth.)
    
    \item \emph{Axiom~S}(table):
    There exist integers $n_k\ge 1$ with
    \begin{equation}
    \label{INEQ:linear-stability-range}
    d(n_k) \le \Cslope k + \Cstart
    \end{equation}
    such that for each $k\ge 0$, $n\ge n_k$, $i\in \set{0,1,\dots,K}$, and $\lambda\in \ptn$ with $l(\lambda) \le c[n]/2$, there exists a $\Gal(\ol{\FF}_q/\FF_q)$-equivariant isomorphism
    \begin{equation*}
    H_k(\mcal{P}_n \otimes \ol{\FF}_q, \det(V_n)^i \otimes \Schur^G_\lambda(V_n)) \cong
    \begin{cases}
        H_k(\mcal{P}_{n_k} \otimes \ol{\FF}_q, \det(V_{n_k})^i \otimes \Schur^G_\lambda(V_{n_k})) & \textnormal{if }l(\lambda) \le c[n_k]/2, \\
        0 & \textnormal{if }l(\lambda) > c[n_k]/2. \\
    \end{cases}
    \end{equation*}
    (If $G=\Sp$, then $\det(V_n) = 1$ can be ignored,
    and\footnote{because $\Schur^{\Sp}_\lambda(V_{n_k}) = 0$ for $l(\lambda) > c[n_k]/2$; cf.~\cite{bergstrom2023hyperelliptic}*{\S11.3.17}.
    For the hyperelliptic ensemble, a much stronger stable $H_k$ vanishing statement, valid for $\abs{\lambda} \gg c[n_k]$ and not just for $l(\lambda) \gg c[n_k]$, is proven in \cite{bergstrom2023hyperelliptic}*{Theorem~7.0.2}.
    We will not need anything so precise, but it is still enlightening.}
    the formula $H_k(\mcal{P}_{n_k} \otimes \ol{\FF}_q, \Schur^G_\lambda(V_{n_k}))$ is valid for all $\lambda$.
    For $G = \Orth$, the more complicated formula above may be the correct generalization,\footnote{See Remark~\ref{RMK:Axiom-SV-vs-S} for further discussion.
    For $G = \GL$, it might be best to replace $\Schur^G_\lambda(V_n)$ and $l(\lambda)$, respectively,
    with $\Schur^G_\lambda(V_n) \otimes \Schur^G_\mu(V_n^\ast)$ and $l(\lambda)+l(\mu)$ for $\lambda,\mu\in \ptn$,
    but this would complicate notation.}
    at least up to adjusting the constants $\Cslope$ and $\Cstart$.)

\end{enumerate}




A priori, for the Ratios Conjecture to make sense, it would be natural to assume two more axioms,
familiar to analytic number theorists (and often already unconditionally proven).
But we will actually prove them
in the cases we need,
under Axioms~L, B, and~S.
\begin{enumerate}
    \item \emph{Axiom~F}(amily):
    The limit
    $$\Cfam \defeq \lim_{n \to \infty} \frac{\card{\mcal{P}_n(\FF_q)}}{q^{d(n)}}$$
    exists in $\RR$.
    
    \item \emph{Axiom~A}(verage):
    For each $\bm{\eps}\in \set{\pm 1}^K$ and $N_1,\dots,N_{K+Q}\in \ZZ_{\ge 0}$,
    the limit
    $$\lim_{n \to \infty} \frac{\sum_{f\in \mcal{P}_n(\FF_q)}
    (\prod_{1\le i\le K:\, \eps_i = 1} A_f(N_i))
    (\prod_{1\le j\le K:\, \eps_j = -1} (-1)^{c(f)} w(f) \ol{A}_f(N_j))
    (\prod_{K+1\le k\le K+Q} B_f(N_k))}{q^{d(n)}}$$
    exists in $\CC$.
    Call this limit $\Cavg(\bm{\eps},\bm{N})$.
    
    
    
    
\end{enumerate}
In Axiom~A, it is more technically convenient to have $q^{d(n)}$ in the denominator than $\card{\mcal{P}_n(\FF_q)}$.
Of course, this choice makes no real difference if Axiom~F holds with $\Cfam > 0$.



In \eqref{EQN:define-Cstab1}, \eqref{EQN:define-Cstab2}, \eqref{EQN:define-Cstab3}, and \eqref{EQN:define-Cstab4}, we will define certain explicit constants $\Cstabone,\Cstabtwo,\Cstabthr,\Cstabfou>0$ in terms of $K$, $Q$, $\Crank$, $\Cbetti$, $\Cslope$, $\Cstart$.
We can now state the main result of the section.

\begin{theorem}
\label{THM:general-axiomatization}
Assume Axioms~L, B, and~S.
If $G=\Orth$ (resp.~$G=\GL$), assume further that $K\le 1$ (resp.~$K=0$).\footnote{For clarity and simplicity, we will only treat general ratios (arbitrary $K,Q\ge 0$) in the case $G=\Sp$.}
For simplicity, assume $c[n] \ge 2Q$ for all $n\ge 1$.
Let
\begin{equation}
\label{INEQ:convenient-q-size-assumption}
\delta \defeq \frac{1}{24 \max(\Cstabtwo, 7(K+Q)\Crank^3\Cslope)},
\qquad\textnormal{and}\qquad
q\ge 2^{12} (2\Cbetti)^{1/(2\delta)}.
\end{equation}
Then Axioms~F and~A hold.
Let $\bm{s}\in \CC^{K+Q}$ lie in the open region
\begin{equation}
\label{INEQ:q-restricted-ratios-region}
\begin{split}
\Re(s_1),\dots,\Re(s_K) &\in (\tfrac12 - \delta, \tfrac12 + \delta), \\
\Re(s_{K+1}),\dots,\Re(s_{K+Q}) &> \tfrac12 + q^{-\delta} + \log_q(\Cbetti').
\end{split}
\end{equation}
Then for all $n\ge 1$, we have
\begin{equation}
\label{INEQ:main-general-goal}
\biggl\lvert{\frac{1}{q^{d(n)}}
\sum_{f\in \mcal{P}_n(\FF_q)}
\frac{\Lambda(s_1,f) \cdots \Lambda(s_K,f)}
{\Lambda(s_{K+1},f) \cdots \Lambda(s_{K+Q},f)}
- \mathsf{MT}(\bm{s};n)}\biggr\rvert
\le 5 q^{\Cstabfou} q^{-d(n)/(7\Crank^2\Cslope)},
\end{equation}
where $\mathsf{MT}(\bm{s};n)$ is a holomorphic function on \eqref{INEQ:q-restricted-ratios-region}, given by the formula\footnote{If $\Cfam > 0$, then $\card{\mcal{P}_n(\FF_q)} \cdot \mathsf{MT}(\bm{s};n)/\Cfam$ is the main term in the Ratios Recipe of \cites{conrey2005integral,conrey2008autocorrelation,andrade2014conjectures},
up to minor details in the handling of averages with root numbers.
In practice, there is complete agreement.}
\begin{equation}
\label{EQN:define-MT-series}
\mathsf{MT}(\bm{s};n)
\defeq \sum_{\bm{\eps}\in \set{\pm 1}^K} \mathsf{F}(\bm{\eps},\bm{s})
\prod_{1\le i\le K:\, \eps_i=-1} (-1)^{c[n]} q^{c[n](1/2-s_i)},
\end{equation}
where $\mathsf{F}(\bm{\eps},\bm{s})$ is a meromorphic function on \eqref{INEQ:q-restricted-ratios-region}, given by the formula
\begin{equation}
\label{EQN:define-F-series}
\mathsf{F}(\bm{\eps},\bm{s})
\defeq \sum_{N_1,\dots,N_{K+Q}\ge 0}
\Cavg(\bm{\eps},\bm{N})
\prod_{1\le i\le K} q^{\eps_i(1/2-s_i)N_i}
\prod_{K+1\le i\le K+Q} q^{(1/2-s_i)N_i}.
\end{equation}
\end{theorem}

\begin{remark}
\label{RMK:Axiom-SV-vs-S}
In place of Axiom~S in Theorem~\ref{THM:general-axiomatization}, one could assume the following:
\begin{itemize}
\item \emph{Axiom~SV}:
There exist constants $\vartheta,\varsigma>0$
such that whenever $0\le k\le \vartheta\cdot d(n) - \varsigma$ and $\lambda\in \ptn$,
there exists a $\Gal(\ol{\FF}_q/\FF_q)$-equivariant isomorphism
$$H_k(\mcal{P}_n \otimes \ol{\FF}_q, \Schur^G_\lambda(V_n))
\cong H_k(\mcal{P}_{n+1} \otimes \ol{\FF}_q, \Schur^G_\lambda(V_{n+1})).$$
\end{itemize}
Axiom~SV implies Axiom~S, and is equivalent for $G=\Sp$.
Axiom~S likely does not imply Axiom~SV in general.
Axiom~SV implies
the following vanishing statement:
\begin{itemize}
    
    \item If $0\le k\le \vartheta\cdot d(n) - \varsigma$
    and $l(\lambda) > \Crank \max(d(1), 2(k+\varsigma)/\vartheta)$, then
    $$H_k(\mcal{P}_n \otimes \ol{\FF}_q, \Schur^G_\lambda(V_n)) = 0.$$
\end{itemize}
This includes $\lambda = (j^{c[n]})$ for $j\ge 1$ and $d(n) > \Crank^2 \max(d(1), 2(k+\varsigma)/\vartheta)$.
So if $G\in \set{\GL,\Orth}$,
then\footnote{by the Grothendieck--Lefschetz trace formula and Deligne's purity theorem
(together with the classical Weyl equidistribution criterion on $\RR/\ZZ$ if $G=\GL$)}
Axiom~SV implies equidistribution of root numbers,
assuming Axiom~B, $q\gg_{\vartheta,\Crank,\Cbetti,\Cbetti'} 1$, and $\Cfam>0$.
However, Axiom~S
is likely a weaker assumption than Axiom~SV,
and it might apply to families where
equidistribution of $w(f)$ fails (as can happen \cites{helfgott2004behaviour,conrad2005root}).
\end{remark}

\begin{remark}

If $K=0$ (the case of ``negative moments'') in Theorem~\ref{THM:general-axiomatization}, then we only need Axiom~S for partitions $\lambda\in \ptn$ with $l(\lambda) \le Q$.\footnote{In particular, the root numbers $w(f)$ discussed in Remark~\ref{RMK:Axiom-SV-vs-S} play no role if $K=0$.}
The combinatorics of the proof also simplifies a lot if $K=0$.
It increases in difficulty with $K$.
\emph{Skew Howe duality}\footnote{See \cite{nazarov2021skew}*{\S\S1--2, especially (1.4) and \S2.2} for a rich historical overview with many references.}
\cite{howe1995perspectives}*{\S\S3.8.7--3.8.9; \S4} (or the associated character identity \cite{adamovich1996tilting}*{Proposition~3.2})
seems to play a crucial role in connecting stable traces to the Ratios Recipe for large $K$.



\end{remark}

\section{General analysis}
\label{SEC:general-analysis}

In this section, we prove Theorem~\ref{THM:general-axiomatization}.
Let assumptions be as in the theorem,
except for the restrictions
\eqref{INEQ:convenient-q-size-assumption} and \eqref{INEQ:q-restricted-ratios-region} on $q$ and $\bm{s}$ (which we will instead gradually introduce).

We need some character and operator notation.
For a finite-dimensional vector space $V$ over a field of characteristic zero, let $\chi^G_\lambda(A)$ be the trace of $A\in G(V)$ on $\Schur^G_\lambda(V)$;
this is an irreducible character of $G(V)$ if $\Schur^G_\lambda(V) \ne 0$,
and the zero function if $\Schur^G_\lambda(V) = 0$.

The trace $\chi^G_\lambda(A)$ depends only on $G$, $\lambda$, and the multiset $\Theta$ of eigenvalues of $A$ on $V$.
(This is clear for $G=\GL$, and the general case follows by induction on $\abs{\lambda}$.)
Let $$\chi^G_\lambda(\Theta) \defeq \chi^G_\lambda(A),
\quad \det(1+x\Theta) \defeq \det(1+xA),$$
for notational convenience.
If the base field of $V$ is $\CC$, let $G(V)_1 \belongs G(V)$ be the set of elements $A\in G(V)$ whose eigenvalues all lie on the complex unit circle.

If $f\in \mcal{P}_n(\FF_q)$, then by \eqref{EQN:completed-L-function-as-a-char-poly} we have
$$\Lambda(s,f) = \prod_{\alpha\in \Theta_f} (1-q^{1/2-s}\alpha),$$
where $\Theta_f$ is the multiset of $c[n]$ eigenvalues $\alpha\in \ol{\QQ}$ of $\Fr_q$ on $(V_n)_f$.
By the purity and pairing assumptions in Axiom~L, the multiset $\Theta_f$ can be realized as the multiset of eigenvalues of some matrix $A\in G(\CC^{c[n]})_1$.
For $\lambda\in \ptn$, let $$\tr_\lambda(n) \defeq
\frac{1}{q^{d(n)}} \sum_{f\in \mcal{P}_n(\FF_q)}
\tr(\Fr_q, \Schur^G_\lambda{(V_n)_f})
= \frac{1}{q^{d(n)}} \sum_{f\in \mcal{P}_n(\FF_q)} \chi^G_\lambda(\Theta_f).$$
Similarly define $\tr_{\lambda \otimes \mu}$ in terms of $\Schur^G_\lambda \otimes \Schur^G_\mu$;
then $\tr_{\lambda \otimes \mu}(n) = q^{-d(n)} \sum_{f\in \mcal{P}_n(\FF_q)} \chi^G_\lambda\chi^G_\mu(\Theta_f)$.


The group $G$ is reductive \cite{goodman2009symmetry}*{Theorem 3.3.11}, though disconnected if $G=\Orth$.
By decomposing $(\bigwedge \CC^{c[n]})^{\otimes K}$ into irreducible rational $G$-representations, we may write
\begin{equation}
\label{EQN:wedge-decompose}
\prod_{1\le i\le K} \det(1+x_iA)
= \prod_{1\le i\le K} \sum_{d\ge 0} x_i^d \tr({\textstyle \bigwedge^d{A}})
= \sum_{\rho\in \ptn} m^0_\rho(\bm{x};n) \chi^G_\rho(A)
\end{equation}
for all $A\in G(\CC^{c[n]})$,
where $m^0_\rho(\bm{x};n)$ is a polynomial in $\bm{x}=(x_1,\dots,x_K)$ with nonnegative integer coefficients.
Similarly, for some polynomials $m^1_\mu \ge 0$ we have
\begin{equation}
\label{EQN:sym-decompose}
\prod_{1\le i\le Q} \det(1-y_iA)^{-1}
= \prod_{1\le i\le Q} \sum_{d\ge 0} y_i^d \tr(\map{Sym}^d{A})
= \sum_{\mu\in \ptn} m^1_\mu(\bm{y};n) \chi^G_\mu(A),
\end{equation}
with absolute convergence for $A\in G(\CC^{c[n]})_1$ and $\abs{y_1},\dots,\abs{y_Q} < 1$.
(To justify absolute convergence, the main point is that $\abs{\chi^G_\mu(A)} \le \dim \Schur^G_\mu(\CC^{c[n]}) = \chi^G_\mu(I_{c[n]})$.)

For later convenience, we record two simple observations.

\begin{proposition}
\label{PROP:m^1_mu-degree-lower-bound}
Each monomial term of $m^1_\mu(\bm{y};n)$ has total degree $\ge \abs{\mu}$.
\end{proposition}

\begin{proof}
This is clear from \eqref{EQN:sym-decompose}
and the fact that $\Hom_G(\Schur^G_\mu(V_n),V_n^{\otimes \abs{\lambda}}) = 0$ for $\abs{\lambda}<\abs{\mu}$.
\end{proof}

\begin{lemma}
\label{LEM:near-zero-optimization}
For reals $x\ge 0$, we have $\min(1,x) \ge 1-e^{-x} \ge \min(1,x) / 2$.
\end{lemma}

\begin{proof}
If $x\ge 1$, then $1 \ge 1-e^{-x} \ge 1-e^{-1} \ge 1/2$.
If $x\in [0,1]$, then $(1-e^{-x})/x$ is maximized at $x=0$ (with value $1$) and minimized at $x=1$ (with value $1-e^{-1} \ge 1/2$).
\end{proof}

Multiplying \eqref{EQN:wedge-decompose} and \eqref{EQN:sym-decompose} with ``$A=\Theta_f$'' for $f\in \mcal{P}_n(\FF_q)$,
summing over $f$,
switching the order of summation (justified by the absolute convergence in \eqref{EQN:wedge-decompose} and \eqref{EQN:sym-decompose}, and the finiteness of the set $\mcal{P}_n(\FF_q)$),
and using the definition of $\tr_{\rho \otimes \mu}(n)$,
we get
\begin{equation}
\label{EQN:expand-ratio-average}
\frac{1}{q^{d(n)}} \sum_{f\in \mcal{P}_n(\FF_q)}
\frac{\prod_{1\le i\le K} \det(1+x_i\Theta_f)}{\prod_{1\le i\le Q} \det(1-y_i\Theta_f)}
= \sum_{\rho,\mu\in \ptn}
m^0_\rho(\bm{x};n)
m^1_\mu(\bm{y};n)
\tr_{\rho \otimes \mu}(n).
\end{equation}
By the Grothendieck--Lefschetz trace formula, we have
\begin{equation}
\label{EQN:initial-LTF}
\tr_{\rho \otimes \mu}(n) = \sum_{k\ge 0} (-1)^k \tr_{\rho \otimes \mu}(n;k),
\end{equation}
where $\tr_{\rho \otimes \mu}(n;k)$ denotes the trace of $\Fr_q$ on $H_k(\mcal{P}_n \otimes \ol{\FF}_q, \Schur^G_\rho(V_n) \otimes \Schur^G_\mu(V_n))$.
Let $\tr_\lambda(n;k) \defeq \tr_{\lambda \otimes \emptyset}(n;k)$.
In order to pass to stable traces in \eqref{EQN:expand-ratio-average},
we need the following lemma.

\begin{lemma}
[Unstable degrees]
\label{LEM:bound-unstable-degrees}
Let $U\in \RR$, $\bm{x}\in \CC^K$, and $\bm{y}\in \CC^Q$.
Then
\begin{equation*}
\sum_{\rho,\mu\in \ptn}
\sum_{k\ge U}
\abs{m^0_\rho(\bm{x};n)
m^1_\mu(\bm{y};n)
\tr_{\rho \otimes \mu}(n;k)}
\le q^{-U/2} \Cbetti^{d(n)} \frac{(2q^\beta)^{c[n] K} 2^{c[n] Q}}{\min(1,\gamma)^{c[n] Q}},
\end{equation*}
provided $\abs{x_i}\le q^\beta$ and $\abs{y_j}\le q^{-\gamma}/\Cbetti'$ with $\beta\ge 0$ and $\gamma>0$.
\end{lemma}

\begin{proof}
By replacing $U$ with $\max(0,U)$, we may assume $U\ge 0$.
Also, since $m^0_\rho$ and $m^1_\mu$ have nonnegative coefficients, we may use the triangle inequality on $m^0_\rho$ and $m^1_\mu$ to reduce to the case $x_i,y_j\in \RR_{\ge 0}$.
But by Deligne's purity theorem and Axiom~B, we have
\begin{equation}
\label{INEQ:Deligne-plus-Betti-bound}
\sum_{k\ge U} \abs{\tr_{\rho \otimes \mu}(n;k)}
\le q^{-U/2} \Cbetti^{d(n)} (\Cbetti')^{\abs{\mu}}
\rank(\Schur^G_\rho(V_n) \otimes \Schur^G_\mu(V_n))
\end{equation}
whenever $\rho_1\le K$ and $l(\mu)\le Q$.
It is known that these are the only relevant $\rho,\mu\in \ptn$;
we will review why later.
But $(\Cbetti')^{\abs{\mu}} m^1_\mu(\bm{y};n) \le m^1_\mu(\Cbetti'\bm{y};n)$ by Proposition~\ref{PROP:m^1_mu-degree-lower-bound}, so we get
\begin{equation*}
\sum_{k\ge U}
\sum_{\rho,\mu\in \ptn}
\abs{m^0_\rho(\bm{x};n)
m^1_\mu(\bm{y};n)
\tr_{\rho \otimes \mu}(n;k)}
\le q^{-U/2} \Cbetti^{d(n)}
\sum_{\rho,\mu\in \ptn}
m^0_\rho(\bm{x};n)
m^1_\mu(\Cbetti'\bm{y};n)\,
\chi^G_\rho \chi^G_\mu(I_{c[n]}).
\end{equation*}
The last sum over $\rho,\mu\in \ptn$
factors as \eqref{EQN:wedge-decompose} times \eqref{EQN:sym-decompose} for $A=1$,
and thus equals
$$\frac{\prod_{1\le i\le K} (1+x_i)^{c[n]}}
{\prod_{1\le i\le Q} (1-\Cbetti'y_i)^{c[n]}}
\le \frac{(1+q^\beta)^{c[n] K}}{(1-q^{-\gamma})^{c[n] Q}}
\le \frac{(2q^\beta)^{c[n] K} 2^{c[n] Q}}{\min(1,\gamma\log{q})^{c[n] Q}}
\le \frac{(2q^\beta)^{c[n] K} 2^{c[n] Q}}{\min(1,\gamma)^{c[n] Q}},$$
where we bound $1-q^{-\gamma}$ from below using Lemma~\ref{LEM:near-zero-optimization} with $x = \gamma\log{q}$.\footnote{The crude nature of this bound may be compared with the first (weaker) stage of the Hadamard three circle theorem approach to bounding $L$-functions near the critical line, used in e.g.~\cite{browning2015rational}*{proof of Lemma~8.4}.
Implicitly, the bound $L(s,\chi_d)^{-1} \ll_{q,\gamma,\eps} \abs{d}^\eps = q^{\eps \deg{d}}$ for $\Re(s) \ge \frac12+\gamma$ captures some nontrivial cancellation over zeros.
(See \cite{bui2021ratios}*{Lemma~5.3} or \cite{florea2021negative}*{Lemma~2.1} for more precise bounds.)}
\end{proof}

At this point the argument gets more complicated for $K\ge 1$ than for $K=0$,
because we need to pass from the unstable entities $m^0_\rho(\bm{x};n)$ and $\Schur_\rho(V_n) \otimes \Schur_\mu(V_n)$ to stable versions thereof.
We begin with some basic results on branching from $\GL$ to $G$.

\begin{lemma}
\label{LEM:stable-GL-to-G-branching-containment}
Let $\lambda,\mu\in \ptn$ with $l(\lambda) \le c[n]/2$.
Suppose $\Hom_G(\Schur^G_\mu(V_n), \Schur_\lambda(V_n)) \ne 0$.
Then $\mu \belongs \lambda$.
In particular, $\mu_1 \le \lambda_1$, $l(\mu) \le l(\lambda)$, and $\abs{\mu} \le \abs{\lambda}$.
\end{lemma}

\begin{proof}
Use the stable Littlewood restriction rules stated in \cite{howe2005stable}*{\S1.3}.
\end{proof}

\begin{lemma}
\label{LEM:stable-tensor-branching-growth-bounds}
Let $\lambda,\mu,\rho\in \ptn$ with $l(\lambda)+l(\mu) \le c[n]/2$.
Suppose $\Hom_G(\Schur^G_\rho(V_n), \Schur_\lambda(V_n) \otimes \Schur_\mu(V_n)) \ne 0$.
Then $l(\rho) \le l(\lambda)+l(\mu)$ and $\abs{\rho} \le \abs{\lambda}+\abs{\mu}$.
\end{lemma}

\begin{proof}
By the Littlewood--Richardson rule \cite{fulton2013representation}*{(A.8)} and the length properties in \cite{goodman2000representations}*{Theorem~9.2.3, (9.2.5)}
concerning the decomposition of $\Schur_\lambda \otimes \Schur_\mu$ into pieces $\Schur_\nu$,
and Lemma~\ref{LEM:stable-GL-to-G-branching-containment} on the branching of $\Schur_\nu$ from $\GL$ to $G$,
there exists $\nu\in \ptn$ such that
$\lambda,\mu,\rho \belongs \nu$,
$\abs{\nu} = \abs{\lambda} + \abs{\mu}$,
and $l(\nu) \le l(\lambda)+l(\mu)$.
\end{proof}

Lemma~\ref{LEM:stable-tensor-branching-growth-bounds} and the stable tensor rules in \cite{howe2005stable}*{\S2.1},
when combined with Axiom~S, imply that for each $\rho,\mu\in \ptn$ and $i\in \set{0,\dots,K}$,
the quantity $\tr_{\det^i \otimes \rho \otimes \mu}(n;k)$ is eventually constant as $n \to \infty$.
(Note that $c[n] \to \infty$ as $n \to \infty$,
by \eqref{INEQ:linear-rank-growth} and \eqref{INEQ:dyadic-dimension-bound-assumption}.)
Let $$T^i_{\rho \otimes \mu}(k) \defeq \lim_{n\to \infty} \tr_{\det^i \otimes \rho \otimes \mu}(n;k).$$
If $G=\Sp$, then $T^i_{\rho \otimes \mu}(k) = T^0_{\rho \otimes \mu}(k)$,
but in general the twist by $\det$ does matter.

Since $c[n] \ge 2Q$, and symmetric powers of $\GL$ correspond to partitions of length $\le 1$,
Lemma~\ref{LEM:stable-tensor-branching-growth-bounds} and \cite{howe2005stable}*{\S1.3}
imply that $m^1_\mu(\bm{y};n)$ is independent of $n$,
and furthermore that $m^1_\mu$ is supported on $\mu\in \ptn$ with $l(\mu) \le Q$.\footnote{Conversely, $m^1_\mu \ne 0$ for all $\mu\in \ptn$ with $l(\mu) \le Q$ (as can be shown by Pieri's formula, or by explicit Cauchy identities), but we will not need this fact.}
Thus for simplicity, we write
\begin{equation*}
m^1_\mu(\bm{y}) \defeq m^1_\mu(\bm{y};n).
\end{equation*}
On the other hand,
we will see that $m^0_\rho$ is supported on $\rho\in \ptn$ with $\rho_1 \le K$ (Lemma~\ref{LEM:skew-multiplicity-formula}(1)),
and that $m^0_\rho(\bm{x};n)$ depends on $n$ in general.

\begin{lemma}
[Ratio trace stability]
\label{LEM:ratio-trace-stability}

Let $\rho,\mu\in \ptn$ with $l(\mu) \le Q$.
Let $k\ge 0$ and $n \ge n_k$ with $c[n] > c[n_k] + 4Q$.
Then the following statements hold:
\begin{enumerate}
    
    \item If $K\ge 1$, $\tr_{\rho \otimes \mu}(n;k) \ne 0$, and $l(\rho) \le c[n]/2$,
    then $$l(\rho) \le c[n_k]/2 + Q < c[n]/2 - Q.$$

    \item If $K\ge 1$, $\tr_{\rho \otimes \mu}(n;k) \ne 0$, and $l(\rho) > c[n]/2$,
    then $G=\Orth$, $\rho^\prime_1+\rho^\prime_2 \le c[n]$ (where $\rho^\prime\in \ptn$ is the conjugate of $\rho$),
    and $$c[n]-l(\rho) \le c[n_k]/2 + Q < c[n]/2 - Q.$$
    
    \item Suppose $l(\rho) \le c[n]/2 - Q$.
    Then $$\tr_{\rho \otimes \mu}(n;k)
    = T^0_{\rho \otimes \mu}(k).$$

    \item Suppose $K\ge 1$, $G=\Orth$, $\rho^\prime_1+\rho^\prime_2 \le c[n]$,
    and $l(\rho) \ge c[n]/2 + Q$.
    Then $$\tr_{\rho \otimes \mu}(n;k)
    = T^1_{\rho^\ast(n) \otimes \mu}(k),$$
    where $\rho^\ast(n)\in \ptn$ is the $\Orth$-associate of $\rho$ (with $\Schur^G_{\rho^\ast(n)} \defeq \det \otimes \Schur^G_\rho$; $\rho^\ast(n)$ matches $\rho$ in all columns except the first, which has length $l(\rho^\ast(n)) = c[n]-l(\rho)$).
\end{enumerate}
\end{lemma}

\begin{proof}

(1)--(2):
Say $K\ge 1$.
Then $G\in \set{\Sp,\Orth}$.
Also suppose $\tr_{\rho \otimes \mu}(n;k) \ne 0$,
with no assumption yet on $l(\rho)$.
Let $\lambda\in \ptn$ with $\tr_\lambda(n;k) \ne 0$ and $\Hom_G(\Schur^G_\lambda(V_n), \Schur^G_\rho(V_n) \otimes \Schur^G_\mu(V_n)) \ne 0$.

\emph{Case~1: $l(\lambda) \le c[n]/2$.}
Then by Axiom~S, $\tr_\lambda(n;k) \ne 0$ implies $l(\lambda) \le c[n_k]/2$.
Thus $$l(\lambda)+l(\mu) \le c[n_k]/2 + Q < c[n]/2 - Q.$$
But by self-duality, $$\Hom_G(\Schur^G_\lambda, \Schur^G_\rho \otimes \Schur^G_\mu)
= (\Schur^G_\lambda \otimes \Schur^G_\rho \otimes \Schur^G_\mu)^G
= \Hom_G(\Schur^G_\rho, \Schur^G_\lambda \otimes \Schur^G_\mu)
\belongs \Hom_G(\Schur^G_\rho, \Schur_\lambda \otimes \Schur_\mu),$$
so by Lemma~\ref{LEM:stable-tensor-branching-growth-bounds}, $l(\rho) \le l(\lambda)+l(\mu)$.
So (1) holds and (2) is vacuous.

\emph{Case~2: $l(\lambda) > c[n]/2$.}
Then $G=\Orth$, since $\Schur^G_\lambda(V_n) \ne 0$ (since $\tr_\lambda(n;k) \ne 0$).
Now $\lambda^\prime_1+\lambda^\prime_2 \le c[n]$ and $\rho^\prime_1+\rho^\prime_2 \le c[n]$, since $\Schur^G_\lambda(V_n) \ne 0$ and $\Schur^G_\rho(V_n) \ne 0$.
Also $l(\lambda^\ast(n)) = c[n] - l(\lambda) < c[n]/2$.
By Axiom~S, $\tr_\lambda(n;k) \ne 0$ implies $l(\lambda^\ast(n)) \le c[n_k]/2$.
But by tensoring $\lambda$ and $\rho$ by $\det$, we have $\Hom_G(\Schur^G_{\lambda^\ast(n)}(V_n), \Schur^G_{\rho^\ast(n)}(V_n) \otimes \Schur^G_\mu(V_n)) \ne 0$,
so arguing as in Case~1, we get
$$l(\rho^\ast(n)) \le l(\lambda^\ast(n))+l(\mu) \le c[n_k]/2 + Q < c[n]/2 - Q.$$
In particular, $l(\rho) > c[n]/2 + Q$.
So (2) holds and (1) is vacuous.

(3):
The assumptions imply $l(\rho)+l(\mu) \le c[n]/2$.
So by Lemma~\ref{LEM:stable-tensor-branching-growth-bounds} and \cite{howe2005stable}*{\S2.1},
we may decompose $\Schur^G_\rho \otimes \Schur^G_\mu$ into pieces $\Schur^G_\nu$ with $l(\nu) \le l(\rho)+l(\mu)$, with multiplicities independent of $n$.
Axiom~S now implies (3).

(4):
$l(\rho^\ast(n))+l(\mu) \le c[n]/2$.
So by Lemma~\ref{LEM:stable-tensor-branching-growth-bounds} and \cite{howe2005stable}*{\S2.1},
we obtain a stable decomposition of $\Schur^G_{\rho^\ast(n)} \otimes \Schur^G_\mu$ into pieces $\Schur^G_\nu$ with $l(\nu) \le c[n]/2$.
Axiom~S now implies (4).
\end{proof}


Let $n_{k,2} \defeq \min\set{n \ge n_k: d(n)/\Crank > \Crank d(n_k) + 4Q+2}$.
Then $n_{k,2} > n_k \ge 1$, so by \eqref{INEQ:dyadic-dimension-bound-assumption},
\begin{equation}
\label{INEQ:bound-dnk2}
d(n_{k,2})
\le 2d(n_{k,2}-1)
\le 2\Crank (\Crank d(n_k) + 4Q+2).
\end{equation}
Furthermore, \eqref{INEQ:linear-rank-growth} implies
\begin{equation}
\label{INEQ:bound-cnk2}
c[n]
\ge d(n)/\Crank - 2
> \Crank d(n_k) + 4Q
\ge c[n_k] + 4Q
\end{equation}
for all $n\ge n_{k,2}$,
since $d(n)$ is weakly increasing in $n$ by \eqref{INEQ:dyadic-dimension-bound-assumption}.

\begin{lemma}
[Stable trace bounds]
\label{LEM:stable-trace-bounds}
Let $\rho,\mu\in \ptn$ with $\rho_1 \le K$ and $l(\mu) \le Q$.
Let $k\ge 0$ and $i\in \set{0,1}$ with $i\le K$.
Then the following statements hold:
\begin{enumerate}
    \item If $T^i_{\rho \otimes \mu}(k) \ne 0$,
    then $l(\rho) \le c[n_k]/2 + Q$.

    \item $\abs{T^i_{\rho \otimes \mu}(k)}
    \le q^{-k/2} \Cbetti^{d(n)} (\Cbetti')^{\abs{\mu}} \chi^G_\rho \chi^G_\mu(I_{c[n]})$
    for all $n\ge n_{k,2}$.
\end{enumerate}
\end{lemma}

\begin{proof}
First say $i=0$.
By Lemma~\ref{LEM:ratio-trace-stability}(3),
we have $T^0_{\rho \otimes \mu}(k) = \tr_{\rho \otimes \mu}(n;k)$ for all $n\ge n_{k,2}$ with $l(\rho) \le c[n]/2 - Q$.
Now (1) follows from Lemma~\ref{LEM:ratio-trace-stability}(1) if $K\ge 1$,
and from $\rho_1 \le K$ if $K=0$.
Finally, (2) follows from (1) if $l(\rho) > c[n_k]/2 + Q$,
and from \eqref{INEQ:Deligne-plus-Betti-bound} if $l(\rho) \le c[n_k]/2 + Q$.

Now suppose $i=1$.
Then $K\ge 1$, so $G\in \set{\Sp,\Orth}$.
If $G=\Sp$ then $T^1 = T^0$ (which we have already treated), so we may assume $G=\Orth$.
For all $n\ge n_{k,2}$ with $l(\rho) \le c[n]/2 - Q$,
Lemma~\ref{LEM:ratio-trace-stability}(4) (with $\rho^\ast(n)$ in place of $\rho$) implies $T^1_{\rho \otimes \mu}(k) = \tr_{\rho^\ast(n) \otimes \mu}(n;k)$.
Now Lemma~\ref{LEM:ratio-trace-stability}(2) implies (1).
Finally, (2) follows from (1) and \eqref{INEQ:Deligne-plus-Betti-bound} as before,
since $\chi^G_{\rho^\ast(n)}(I_{c[n]}) = \chi^G_\rho(I_{c[n]})$.
\end{proof}

\begin{definition}
[Stable skew multiplicity]
\label{DEFN:stable-skew-multiplicity}
Let $\rho\in \ptn$ with $\rho_1 \le K$.
If $G=\Sp$, let
\begin{equation*}
M^0_\rho(\bm{x};n)
\defeq \frac{(x_1\cdots x_K)^{c[n]/2+K}
\det(x_j^{\rho^\prime_i-c[n]/2-i} - x_j^{-(\rho^\prime_i-c[n]/2-i)})_{1\le i,j\le K}}
{D_K(\bm{x})}
\in \ZZ[x_1^{\pm 1},\dots,x_K^{\pm 1}],
\end{equation*}
where
\begin{equation*}
D_K(\bm{x})\defeq \prod_{1\le i<j\le K} (x_i-x_j)(x_ix_j-1) \prod_{1\le i\le K} (1-x_i^2).
\end{equation*}
If $K=0$, let $M^0_\rho(\bm{x};n) \defeq 1$.
If $G=\Orth$ and $K=1$, let
\begin{equation*}
M^0_\rho(\bm{x};n)
\defeq x_1^{l(\rho)}.
\end{equation*}
For later convenience, also let $M^0_\rho(\bm{x};n;0) \defeq M^0_\rho(\bm{x};n)$ and
\begin{equation*}
M^0_\rho(\bm{x};n;1) \defeq
\begin{cases}
    x_1^{c[n]-l(\rho)} & \textnormal{if }(G,K)=(\Orth,1), \\
    0 & \textnormal{otherwise}.
\end{cases}
\end{equation*}

\end{definition}

\begin{lemma}
[Skew multiplicity formula]
\label{LEM:skew-multiplicity-formula}
Let $\rho\in \ptn$.
\begin{enumerate}
    \item If $\rho_1>K$, then $m^0_\rho(\bm{x};n) = 0$.

    \item Suppose $\rho_1\le K$ and $l(\rho) \le c[n]/2$.
    Then $m^0_\rho(\bm{x};n) = M^0_\rho(\bm{x};n;0) \ne 0$.
    If moreover $G=\Orth$ and $K=1$, then $m^0_{\rho^\ast(n)}(\bm{x};n) = M^0_\rho(\bm{x};n;1) \ne 0$.
\end{enumerate}
\end{lemma}

\begin{proof}

\emph{Case~1: $K=0$.}
Then (1) is clear, since $m^0_\rho$ is supported on $\rho=\emptyset$.
In (2) we have $\rho=\emptyset$ (since $\rho_1\le 0$),
so $m^0_\rho(\bm{x};n) = 1 = M^0_\rho(\bm{x};n;0)$.

\emph{Case~2: $(G,K)=(\Orth,1)$.}
By \cite{goodman2009symmetry}*{Corollary 5.5.6}, the $G$-representations $\bigwedge^i \CC^{c[n]}$ are irreducible and pairwise non-isomorphic for $0\le i\le c[n]$.
But for $0\le i\le c[n]$ we have $\Schur^G_{(1^i)} \CC^{c[n]} \ne 0$ and $\Schur_{(1^i)} \CC^{c[n]} = \bigwedge^i \CC^{c[n]}$, and thus $\Schur^G_{(1^i)} \CC^{c[n]} = \bigwedge^i \CC^{c[n]}$.
Now (1) is clear, and in (2) we have
$m^0_\rho(\bm{x};n) = x_1^{l(\rho)} = M^0_\rho(\bm{x};n;0)$
and $m^0_{\rho^\ast(n)}(\bm{x};n) = x_1^{l(\rho^\ast(n))} = M^0_\rho(\bm{x};n;1)$.

\emph{Case~3: $G=\Sp$.}
Skew Howe duality for $\Sp$ (see \cite{bump2006averages}*{Remark~7 and references within}) implies the following:
$m^0_\rho = 0$ unless $\rho \belongs (K^{c[n]/2})$, in which case
\begin{equation*}
m^0_\rho(\bm{x};n)
= (x_1\cdots x_K)^{c[n]/2}
\chi^G_{(c[n]/2-\rho^\prime_K,\dots,c[n]/2-\rho^\prime_1)}(x_1^{\pm 1},\dots,x_K^{\pm 1})
\ne 0.
\end{equation*}
By the Weyl character formula for $\Sp$ \cite{bump2006averages}*{(43)}, the last character is
\begin{equation*}
\frac{\det(x_j^{c[n]/2-\rho^\prime_{K-i+1}+K-i+1} - x_j^{-(c[n]/2-\rho^\prime_{K-i+1}+K-i+1)})_{1\le i,j\le K}}
{\det(x_j^{K-i+1} - x_j^{-(K-i+1)})_{1\le i,j\le K}}
= \frac{\det(x_j^{\rho^\prime_i-c[n]/2-i} - x_j^{-(\rho^\prime_i-c[n]/2-i)})_{1\le i,j\le K}}
{\det(x_j^{-i} - x_j^i)_{1\le i,j\le K}}
\end{equation*}
(where to equate the two fractions, we swap rows $i$ and $K-i+1$ in the $K\times K$ determinants, and multiply each column by $-1$).
By the Weyl denominator formula for $\Sp$ \cite{bump2006averages}*{(44)},
$$\det(x_j^{-i} - x_j^i)_{1\le i,j\le K}
= (x_1\cdots x_K)^{-K} D_K(\bm{x}).$$
This completes the proof of (1)--(2).
\end{proof}

\begin{lemma}
[Stable multiplicity bound]
\label{LEM:stable-multiplicity-bound}
Let $\rho\in \ptn$ with $\rho_1 \le K$.
Then $$\abs{M^0_\rho(\bm{x};n;0)}
\le \abs{x_1\cdots x_K}^{c[n]/2}
\prod_{1\le i\le K} (\abs{x_i}+\abs{x_i}^{-1}+2)^{c[n]/2 + l(\rho)}.$$
If $(G,K)=(\Orth,1)$, then
$\abs{M^0_\rho(\bm{x};n;1)}$ is also at most the quantity on the right.
\end{lemma}

\begin{proof}
\emph{Case~1: $K=0$.}
Here $\abs{M^0} \le 1$, by Definition~\ref{DEFN:stable-skew-multiplicity}.
This suffices.

\emph{Case~2: $(G,K)=(\Orth,1)$.}
Use Definition~\ref{DEFN:stable-skew-multiplicity} and note that
$$\abs{x_1}^d \le \abs{x_1}^{c[n]/2} \max(\abs{x_1},\abs{x_1}^{-1})^{c[n]/2 + l(\rho)}$$
for all $d\in [-l(\rho), c[n]+l(\rho)]$.
Taking $d\in \set{l(\rho), c[n]-l(\rho)}$ gives the result.

\emph{Case~3: $G=\Sp$.}
Sort the absolute values $\abs{\rho^\prime_i-c[n]/2-i}$ for $1\le i\le K$
into a non-decreasing list of nonnegative integers, $u_1\le \cdots \le u_K$.
By Definition~\ref{DEFN:stable-skew-multiplicity},
the ratio $\abs{M^0_\rho(\bm{x};n;0) / (x_1\cdots x_K)^{c[n]/2}}$
vanishes unless $0<u_1<\dots<u_K$,
in which case it equals $\abs{\chi^G_{(u_K-K,\dots,u_1-1)}(x_1^{\pm 1},\dots,x_K^{\pm 1})}$ by the Weyl character formula.
By \cite{bump2006averages}*{Remark~7, (50)}
and the fact that symplectic Schur polynomials have \emph{nonnegative} integer coefficients (with a combinatorial interpretation; see \cite{sundaram1990tableaux}*{Theorem~2.3} or \cite{krattenthaler1998identities}*{(A.3)--(A.4)}),
we have
\begin{equation*}
\abs{\chi^G_{\lambda}(x_1^{\pm 1},\dots,x_K^{\pm 1})}
\le \prod_{1\le i\le K} \prod_{1\le j\le N} (\abs{x_i}+\abs{x_i}^{-1}+1+1^{-1})
\end{equation*}
whenever $\lambda\belongs (N^K)$.
Since $-c[n]/2-K \le \rho^\prime_i-c[n]/2-i \le l(\rho)-c[n]/2-1$ for $1\le i\le K$,
we have $u_K-K \le \max(c[n]/2, l(\rho)) \le c[n]/2 + l(\rho)$.
Taking $N = c[n]/2 + l(\rho)$ completes the proof.
(A comparable bound on $M^0$ in terms of eigenvalues and dimensions, not using positivity, is also possible; cf.~\cite{bergstrom2023hyperelliptic}*{Lemma~11.3.14}.)
\end{proof}

For convenience, let
\begin{align}
\Cstabone &\defeq 2 \Crank \max(1, (K+Q)\Crank) \ge 2,
\label{EQN:define-Cstab1} \\
\Cstabtwo &\defeq K \Crank \Cslope + \Cstabone \Crank \Cslope \ge 2,
\label{EQN:define-Cstab2} \\
\Cstabthr &\defeq K (\Crank \Cstart + 2Q) + \Cstabone (\Crank \Cstart + 4Q+2) \ge6.
\label{EQN:define-Cstab3}
\end{align}

\begin{lemma}
[Missing stable parts]
\label{LEM:stable-gap-bounds}
Let $U\in \RR$, $\bm{x}\in \CC^K$, and $\bm{y}\in \CC^Q$.
Assume
\begin{equation}
\label{INEQ:q-conditions-try-1}
q^{-\beta} \le \abs{x_u} \le q^\beta,
\quad \abs{y_v}\le q^{-\gamma}/\Cbetti',
\quad q^{1/6} \min(1,\gamma)^{\Cstabtwo} \ge 2 (2 q^\beta \Cbetti)^{\Cstabtwo},
\end{equation}
where $\beta\ge 0$ and $\gamma>0$.
Then
\begin{equation*}
\sum_{0\le i\le 1}
\sum_{\substack{\rho,\mu\in \ptn: \\ \rho_1\le K,\; l(\mu)\le Q}}
\sum_{k\ge U}
\abs{M^0_\rho(\bm{x};n;i)
m^1_\mu(\bm{y})
T^i_{\rho \otimes \mu}(k)}
\le 4 (2q^\beta)^{K c[n]}
(q^{1/6}/2)^{\Cstabthr/\Cstabtwo} q^{-U/3}.
\end{equation*}
\end{lemma}

\begin{proof}
By replacing $U$ with $\max(U,0)$ if necessary, we may assume $U\ge 0$.
By the triangle inequality on $m^1$, we may also assume $y_v = q^{-\gamma}/\Cbetti'$ for $1\le v\le Q$.
Next, note that if $(i,\rho,\mu,k)$ has a nonzero contribution to the sum on the left,
then by Lemma~\ref{LEM:stable-trace-bounds}(1),
we have $l(\rho) \le c[n_k]/2 + Q$.
Using Lemma~\ref{LEM:stable-multiplicity-bound} to bound $\abs{M^0}$
and Lemma~\ref{LEM:stable-trace-bounds}(2) to bound $\abs{T^i}$,
we find that the total quantity on the left is
(by Proposition~\ref{PROP:m^1_mu-degree-lower-bound})
\begin{equation*}
\le 2 \sum_{\substack{k\ge U, \\ \rho\belongs (K^{c[n_k]/2 + Q}), \\ l(\mu)\le Q}}
q^{\beta K c[n]/2}
(4q^\beta)^{K (c[n]/2 + c[n_k]/2 + Q)}
m^1_\mu(\Cbetti'\bm{y})
q^{-k/2} \Cbetti^{d(n_{k,2})} \chi^G_\rho \chi^G_\mu(I_{c[n_{k,2}]}).
\end{equation*}

Here $c[n_k]/2 + Q < c[n_{k,2}]/2 - Q$ by \eqref{INEQ:bound-cnk2}.
So $m^0_\rho(1,\dots,1;n_{k,2}) \ge 1$,
by Lemma~\ref{LEM:skew-multiplicity-formula}(2) (since $m^0_\rho(\bm{x};n_{k,2})\in \ZZ[\bm{x}]$ has nonnegative coefficients).
Therefore, for each $k\ge U$,
\begin{equation*}
\begin{split}
\sum_{\substack{\rho\belongs (K^{c[n_k]/2 + Q}), \\ l(\mu)\le Q}}
m^1_\mu(\Cbetti'\bm{y}) \chi^G_\rho \chi^G_\mu(I_{c[n_{k,2}]})
&\le \sum_{\rho,\mu\in \ptn}
m^0_\rho(1,\dots,1;n_{k,2}) m^1_\mu(\Cbetti'\bm{y}) \chi^G_\rho \chi^G_\mu(I_{c[n_{k,2}]}) \\
&= \frac{\prod_{1\le i\le K} (1+1)^{c[n_{k,2}]}}
{\prod_{1\le i\le Q} (1-q^{-\gamma})^{c[n_{k,2}]}}
\le \frac{2^{c[n_{k,2}] K} 2^{c[n_{k,2}] Q}}{\min(1,\gamma)^{c[n_{k,2}] Q}},
\end{split}
\end{equation*}
by \eqref{EQN:wedge-decompose}, \eqref{EQN:sym-decompose}, and Lemma~\ref{LEM:near-zero-optimization}; cf.~the proof of Lemma~\ref{LEM:bound-unstable-degrees}.
Plugging the last display into the penultimate display and summing over $k\ge U$, we get a total bound of
\begin{equation*}
\le 2 \sum_{k\ge U}
q^{\beta K c[n]/2}
(4q^\beta)^{K (c[n]/2 + c[n_k]/2 + Q)}
q^{-k/2} \Cbetti^{d(n_{k,2})}
\frac{2^{c[n_{k,2}] K} 2^{c[n_{k,2}] Q}}{\min(1,\gamma)^{c[n_{k,2}] Q}}.
\end{equation*}

By \eqref{INEQ:linear-stability-range}, $d(n_k)\le \Cslope k + \Cstart$.
By \eqref{INEQ:bound-dnk2} we then get $d(n_{k,2}) \le 2\Crank (\Crank (\Cslope k + \Cstart) + 4Q+2)$.
Therefore,
\begin{equation*}
K (\Crank d(n_k) + 2Q) + \max(1, \Crank (K+Q)) d(n_{k,2})
\le \Cstabtwo k + \Cstabthr.
\end{equation*}
Yet by \eqref{INEQ:linear-rank-growth}, $c[n_k] \le \Crank d(n_k)$
and $c[n_{k,2}] \le \Crank d(n_{k,2})$.
Thus the penultimate display is
\begin{equation*}
\le 2 \sum_{k\ge U}
(2q^\beta)^{K c[n]}
q^{-k/2}
\frac{(2 q^\beta \Cbetti)^{\Cstabtwo k + \Cstabthr}}
{\min(1,\gamma)^{\Cstabtwo k + \Cstabthr}}
\le 2 (2q^\beta)^{K c[n]} \sum_{k\ge U}
q^{-k/2} (q^{1/6}/2)^{k + \Cstabthr/\Cstabtwo},
\end{equation*}
since $2 q^\beta \Cbetti \le (q^{1/6}/2)^{1/\Cstabtwo} \min(1,\gamma)$.
The final geometric series over $k$ has ratio $(2q^{1/3})^{-1} \le \frac12$,
and is therefore $\le (q^{1/6}/2)^{\Cstabthr/\Cstabtwo} (2q^{1/3})^{-U} / (1-\frac12)^{-1}$.
\end{proof}

\begin{lemma}
\label{LEM:stable-interchange}
Assume $q^{1/6}\ge 2 (2 \Cbetti)^{\Cstabtwo}$.
Fix $\rho,\mu\in \ptn$ with $\rho_1 \le K$
and $m^1_\mu \ne 0$.
Let $0\le i\le K$.
Then $T^i_{\rho \otimes \mu}
\defeq \lim_{n \to \infty} \tr_{\det^i \otimes \rho \otimes \mu}(n)$
exists, and equals $\sum_{k\ge 0} (-1)^k T^i_{\rho \otimes \mu}(k)$.
\end{lemma}

\begin{proof}
We may assume $i=0$ if $G\in \set{\GL,\Sp}$, and $i\le 1$ if $G=\Orth$.
We may also assume $n$ is large enough so that $c[n] \ge 2 l(\rho)$.
Then by Lemma~\ref{LEM:skew-multiplicity-formula}(2), we have $M^0_\rho(\bm{x};n;i) = m^0_{\det^i \otimes \rho}(\bm{x};n) \ge 1$.
Applying Lemma~\ref{LEM:stable-gap-bounds} with $x_1=\dots=x_K=1$, $y_1=\dots=y_Q=q^{-\gamma}/\Cbetti'$, $\beta=0$, $\gamma=1$, and the smallest $n\ge 1$ for which $c[n] \ge 2 l(\rho)$, we get
\begin{equation}
\label{INEQ:blah-1}
\sum_{k\ge U}
\abs{T^i_{\rho \otimes \mu}(k)}
\ll_{q,\rho,\mu} q^{-U/3};
\end{equation}
in particular, the infinite sum $\sum_{k\ge 0} (-1)^k T^i_{\rho \otimes \mu}(k)$ converges absolutely.
Similarly, via Lemma~\ref{LEM:bound-unstable-degrees} with $x_1=\dots=x_K=1$ and $y_1=\dots=y_Q=q^{-\gamma}/\Cbetti'$, we have
\begin{equation}
\label{INEQ:blah-2}
\sum_{k\ge U}
\abs{\tr_{\det^i \otimes \rho \otimes \mu}(n;k)}
\ll_{q,\rho,\mu} q^{-U/2} \Cbetti^{d(n)} 2^{c[n] (K+Q)}
\end{equation}
(for all $n \gg_\rho 1$).
Yet Lemma~\ref{LEM:ratio-trace-stability}(3)--(4) imply
(as noted in the proof of Lemma~\ref{LEM:stable-trace-bounds})
that $T^i_{\rho \otimes \mu}(k) = \tr_{\det^i \otimes \rho \otimes \mu}(n;k)$ for all $n\ge n_{k,2}$ with $l(\rho) \le c[n]/2 - Q$.
Let $$u_n \defeq \frac{d(n)}{2\Crank^2\Cslope}
- \frac{\Crank\Cstart + 4Q+2}{\Crank\Cslope}
= \frac{d(n)}{2\Crank^2\Cslope} + O_Q(1)$$
be the unique real solution to $d(n) = 2\Crank (\Crank (\Cslope u_n + \Cstart) + 4Q+2)$.
Then for all (nonnegative integers) $k < u_n$,
we have $d(n) > d(n_{k,2})$ by \eqref{INEQ:bound-dnk2} and \eqref{INEQ:linear-stability-range}, whence $n > n_{k,2}$ by \eqref{INEQ:dyadic-dimension-bound-assumption}.
So \eqref{EQN:initial-LTF},
combined with \eqref{INEQ:blah-1} and \eqref{INEQ:blah-2} for $U=u_n$,
yields the inequality
\begin{equation*}
\tr_{\det^i \otimes \rho \otimes \mu}(n)
- \sum_{k\ge 0} (-1)^k T^i_{\rho \otimes \mu}(k)
\ll_{q,\rho,\mu} q^{-u_n/3} \Cbetti^{d(n)} 2^{c[n] (K+Q)}.
\end{equation*}
The right-hand side tends to $0$ as $n\to \infty$,
because $c[n]\le \Crank d(n)$ (by \eqref{INEQ:linear-rank-growth}) and
$q^{1/(6\Crank^2 \Cslope)} > (2\Cbetti)^{\Cstabtwo/(\Crank^2 \Cslope)}
\ge \Cbetti 2^{\Crank (K+Q)}$ (since $\Cstabtwo \ge \Cstabone \Crank \Cslope$).
Lemma~\ref{LEM:stable-interchange} follows.
\end{proof}

We now prove our main result
on traces.
For convenience, let
\begin{equation}
\label{EQN:define-Cstab4}
\Cstabfou
\defeq \frac{\Cstabthr}{6\Cstabtwo}
+ \frac{\Crank\Cstart + 4Q+2}{2\Crank\Cslope}.
\end{equation}

\begin{lemma}
\label{LEM:trace-formulation-of-main-theorem}
Let $\bm{x}\in \CC^K$ and $\bm{y}\in \CC^Q$.
Assume \eqref{INEQ:q-conditions-try-1},
where $\beta\ge 0$ and $\gamma>0$.
Then
\begin{equation*}
\biggl\lvert{\sum_{\rho,\mu\in \ptn}
m^0_\rho(\bm{x};n)
m^1_\mu(\bm{y})
\tr_{\rho \otimes \mu}(n)
- \mathsf{ST}(\bm{x},\bm{y};n)}\biggr\rvert
\le 5 q^{\Cstabfou}
\frac{(2 q^\beta \Cbetti)^{\Crank d(n) (K+Q)}}{\min(1,\gamma)^{\Crank d(n) (K+Q)}}
q^{-d(n)/(6\Crank^2\Cslope)},
\end{equation*}
where
\begin{equation*}
\mathsf{ST}(\bm{x},\bm{y};n)
\defeq \sum_{0\le i\le 1}
\sum_{\substack{\rho,\mu\in \ptn: \\ \rho_1\le K,\; l(\mu)\le Q}}
M^0_\rho(\bm{x};n;i)
m^1_\mu(\bm{y})
T^i_{\rho \otimes \mu}.
\end{equation*}
\end{lemma}

\begin{proof}
The strategy is similar to, but subtler than, that for the proof of Lemma~\ref{LEM:stable-interchange}.
Let $u_n$ be defined as in that proof.
By \eqref{EQN:initial-LTF},
\begin{equation*}
\sum_{\rho,\mu\in \ptn}
m^0_\rho(\bm{x};n)
m^1_\mu(\bm{y})
\tr_{\rho \otimes \mu}(n)
= \sum_{\rho,\mu\in \ptn}
m^0_\rho(\bm{x};n)
m^1_\mu(\bm{y})
\sum_{k\ge 0} (-1)^k \tr_{\rho \otimes \mu}(n;k).
\end{equation*}
We now analyze $k\ge u_n$ and $k<u_n$ separately.
First, by Lemma~\ref{LEM:bound-unstable-degrees},
the total contribution from $k\ge u_n$ has absolute value at most
\begin{equation}
\label{INEQ:blah-3}
q^{-u_n/2} \Cbetti^{d(n)} \frac{(2q^\beta)^{c[n] K} 2^{c[n] Q}}{\min(1,\gamma)^{c[n] Q}}
\le q^{\Cstabfou - d(n)/(4\Crank^2\Cslope)}
\frac{(2 q^\beta \Cbetti)^{\Crank d(n) (K+Q)}}{\min(1,\gamma)^{\Crank d(n) (K+Q)}}.
\end{equation}

Second, by Lemma~\ref{LEM:skew-multiplicity-formula}, the total contribution from $k<u_n$ may be written as
\begin{equation*}
\sum_{0\le i\le 1}
\sum_{\substack{\rho,\mu\in \ptn: \\ \rho_1\le K,\; l(\mu)\le Q, \\ l(\rho) \le (c[n]-i)/2}}
M^0_\rho(\bm{x};n;i)
m^1_\mu(\bm{y})
\sum_{k<u_n} (-1)^k \tr_{\det^i \otimes \rho \otimes \mu}(n;k).
\end{equation*}
By Lemma~\ref{LEM:ratio-trace-stability} (which applies since $n\ge n_{k,2}$ for $k<u_n$), this quantity equals
\begin{equation*}
\sum_{0\le i\le 1}
\sum_{\substack{\rho,\mu\in \ptn: \\ \rho_1\le K,\; l(\mu)\le Q, \\ l(\rho) \le (c[n]-1)/2 - Q}}
M^0_\rho(\bm{x};n;i)
m^1_\mu(\bm{y})
\sum_{k<u_n} (-1)^k T^i_{\rho \otimes \mu}(n;k).
\end{equation*}
Next, we remove the length restriction on $l(\rho)$.
If $i\in \set{0,1}$ and $l(\rho) > (c[n]-1)/2 - Q$, then $l(\rho) \ge c[n]/2 - Q > c[n_k]/2 + Q$ by \eqref{INEQ:bound-cnk2}, so
$M^0_\rho(\bm{x};n;i) T^i_{\rho \otimes \mu}(k) = 0$ by Lemma~\ref{LEM:stable-trace-bounds}(1) and Definition~\ref{DEFN:stable-skew-multiplicity}.
Therefore, the last display equals
\begin{equation*}
\sum_{0\le i\le 1}
\sum_{\substack{\rho,\mu\in \ptn: \\ \rho_1\le K,\; l(\mu)\le Q}}
M^0_\rho(\bm{x};n;i)
m^1_\mu(\bm{y})
\sum_{k<u_n} (-1)^k T^i_{\rho \otimes \mu}(n;k).
\end{equation*}
We would like to complete the sum over $k$.
By Lemma~\ref{LEM:stable-gap-bounds}, the analog of the last display over $k\ge u_n$ (instead of $k<u_n$) has absolute value at most
\begin{equation}
\label{INEQ:blah-4}
4 (2q^\beta)^{K c[n]}
(q^{1/6}/2)^{\Cstabthr/\Cstabtwo} q^{-u_n/3}
\le 4 (2q^\beta)^{K \Crank d(n)}
q^{\Cstabfou - d(n)/(6\Crank^2\Cslope)}.
\end{equation}
Using Lemma~\ref{LEM:stable-interchange} to write $T^i_{\rho \otimes \mu} = \sum_{k\ge 0} (-1)^k T^i_{\rho \otimes \mu}(k)$ in $\mathsf{ST}(\bm{x},\bm{y};n)$,
we find that $\mathsf{ST}(\bm{x},\bm{y};n)$ differs from the penultimate display by at most the values in the last display.

Adding up the error terms recorded in \eqref{INEQ:blah-3} and \eqref{INEQ:blah-4},
we get Lemma~\ref{LEM:trace-formulation-of-main-theorem}.
\end{proof}

To pass from $\mathsf{ST}(\bm{x},\bm{y};n)$ to the actual main term $\mathsf{MT}(\bm{s};n)$ (from \eqref{EQN:define-MT-series}) in the Ratios Conjecture, a bit more work is needed.
This task is simpler than computing all the $T^i_\lambda(k)$;
it can often be done analytically (as noted in \cite{bergstrom2023hyperelliptic}*{\S11.2.4} when $Q=0$).
To make the required comparison in general, we decompose the ``numerators'' of the stable skew multiplicities $M^0_\rho$ from Definition~\ref{DEFN:stable-skew-multiplicity} into simpler pieces;
the rough strategy is similar to that of multiplying by a Vandermonde determinant in \cite{sawin2020representation}.

\begin{definition}
[Stable numerators]
\label{DEFN:stable-numerators}
Let $\rho\in \ptn$ with $\rho_1 \le K$.
For $i\in \set{0,1}$ and $\bm{\eps}\in \set{\pm 1}^K$, let
$N^{\bm{\eps}}_\rho(\bm{x};i) \in \ZZ[x_1^{\pm 1},\dots,x_K^{\pm 1}]$
be the coefficient of $\prod_{1\le u\le K} x_u^{(1-\eps_u) c[n]/2}$
in the product
\begin{equation}
\label{EQN:define-full-stable-numerator}
N_\rho(\bm{x};n;i) \defeq M^0_\rho(\bm{x};n;i) D_K(\bm{x}),
\end{equation}
when $N_\rho(\bm{x};n;i)$ is viewed as a linear recurrence in $c[n]$.
\end{definition}

By Definitions~\ref{DEFN:stable-skew-multiplicity} and~\ref{DEFN:stable-numerators}, it is clear that
\begin{equation}
\label{EQN:expand-stable-numerator}
N_\rho(\bm{x};n;i)
= \sum_{\bm{\eps}\in \set{\pm 1}^K} N^{\bm{\eps}}_\rho(\bm{x};i)
\prod_{1\le u\le K} x_u^{(1-\eps_u) c[n]/2},
\end{equation}
and that with respect to the variables $x_u^{\eps_u}$,
the total degree of any monomial term of the Laurent polynomial $N^{\bm{\eps}}_\rho(\bm{x};i) \in \ZZ[x_1^{\pm 1},\dots,x_K^{\pm 1}]$ lies in the interval $[\abs{\rho}-C_K,\abs{\rho}+C_K]$,
for some constant $C_K\ge 0$ depending only on $K$.
Also recall, from Proposition~\ref{PROP:m^1_mu-degree-lower-bound}, that each monomial term of $m^1_\mu(\bm{y})$ has total degree $\ge \abs{\mu}$.

Given $\bm{\eps}\in \set{\pm 1}^K$ and $\bm{d}=(d_1,\dots,d_{K+Q})\in \ZZ^{K+Q}$,
let $\Cavg^\ast(\bm{\eps},\bm{d};n)$ be the coefficient of $$\prod_{1\le u\le K} x_u^{(1-\eps_u) c[n]/2 + \eps_u d_u} \prod_{K+1\le u\le K+Q} y_{u-K}^{d_u}$$ when the left-hand side of \eqref{EQN:expand-ratio-average} is expanded as a power series in $\bm{x}$ and $\bm{y}$.
Let $$F^\ast(\bm{\eps},\bm{x},\bm{y};n) \defeq
\sum_{\bm{d} \in \ZZ^{K+Q}} \Cavg^\ast(\bm{\eps},\bm{d};n) \prod_{1\le u\le K} x_u^{\eps_u d_u} \prod_{K+1\le u\le K+Q} y_{u-K}^{d_u}
\in \CC[[x_u^{\eps_u}, \bm{y}]].$$
If $q^{1/6}\ge 2 (2 \Cbetti)^{\Cstabtwo}$,
then $T^i_{\rho \otimes \mu}$ exists by Lemma~\ref{LEM:stable-interchange},
and we may define
\begin{equation*}
\mathsf{SN}(\bm{\eps},\bm{x},\bm{y})
\defeq \sum_{0\le i\le 1}
\sum_{\substack{\rho,\mu\in \ptn: \\ \rho_1\le K,\; l(\mu)\le Q}}
N^{\bm{\eps}}_\rho(\bm{x};i)
m^1_\mu(\bm{y})
T^i_{\rho \otimes \mu}
\in \CC[[x_u^{\eps_u}, \bm{y}]][x_u^{-\eps_u},\bm{y}^{-1}].
\end{equation*}
Given a sequence of formal Laurent series $\mcal{L}_n \in \CC[[x_u^{\eps_u}, \bm{y}]][x_u^{-\eps_u},\bm{y}^{-1}]$
indexed by $n\ge 1$, let $\lim_{n\to \infty} \mcal{L}_n$ be the coefficient-wise limit of $\mcal{L}_n$ in $\CC[[x_u^{\eps_u}, \bm{y}]][x_u^{-\eps_u},\bm{y}^{-1}]$, if it exists.

\begin{lemma}
[Formal calculations]
\label{LEM:key-formal-coefficient-calculation}
Assume $q^{1/6}\ge 2 (2 \Cbetti)^{\Cstabtwo}$.
Let $\bm{e}=(e_1,\dots,e_K)\in \set{\pm 1}^K$.
\begin{enumerate}
    \item $\lim_{n\to \infty} D_K(\bm{x}) F^\ast(\bm{e},\bm{x},\bm{y};n)$
    exists, and equals $\mathsf{SN}(\bm{e},\bm{x},\bm{y})$.

    \item $\lim_{n\to \infty} F^\ast(\bm{e},\bm{x},\bm{y};n)$ exists.

    \item Axioms~F and~A hold,
    and $\mathsf{F}(\bm{e},\bm{s})
    = \lim_{n\to \infty} F^\ast(\bm{e},\bm{x},\bm{y};n)$ under the substitution
    \begin{equation}
    \label{EQN:xy-to-s-substitution}
    x_u = -q^{1/2-s_u}\quad (1\le u\le K),
    \qquad y_v = q^{1/2-s_{K+v}}\quad (1\le v\le Q).
    \end{equation}

    \item $D_K(\bm{x}) \mathsf{F}(\bm{e},\bm{s}) = \mathsf{SN}(\bm{e},\bm{x},\bm{y})$
    in the ring $\CC[[x_u^{\eps_u}, \bm{y}]][x_u^{-\eps_u},\bm{y}^{-1}]$.
\end{enumerate}
\end{lemma}

\begin{proof}
(1):
Fix $\bm{d}=(d_1,\dots,d_{K+Q})\in \ZZ^{K+Q}$.
By \eqref{EQN:expand-ratio-average} and the definition of $F^\ast$, we have
\begin{equation*}
F^\ast(\bm{\eps},\bm{x},\bm{y};n) \prod_{1\le u\le K} x_i^{(1-e_u) c[n]/2}
= \sum_{\rho,\mu\in \ptn}
m^0_\rho(\bm{x};n)
m^1_\mu(\bm{y})
\tr_{\rho \otimes \mu}(n).
\end{equation*}
By Lemma~\ref{LEM:skew-multiplicity-formula} and \eqref{EQN:define-full-stable-numerator},
\begin{equation*}
D_K(\bm{x}) \sum_{\rho,\mu\in \ptn}
m^0_\rho(\bm{x};n)
m^1_\mu(\bm{y})
\tr_{\rho \otimes \mu}(n)
= \sum_{0\le i\le 1}
\sum_{\substack{\rho,\mu\in \ptn: \\ \rho_1\le K,\; l(\mu)\le Q, \\ l(\rho) \le (c[n]-i)/2}}
N_\rho(\bm{x};n;i)
m^1_\mu(\bm{y})
\tr_{\det^i \otimes \rho \otimes \mu}(n).
\end{equation*}
By \eqref{EQN:expand-stable-numerator}, the last expression equals
\begin{equation*}
\sum_{0\le i\le 1}
\sum_{\substack{\rho,\mu\in \ptn: \\ \rho_1\le K,\; l(\mu)\le Q, \\ l(\rho) \le (c[n]-i)/2}}
m^1_\mu(\bm{y})
\tr_{\det^i \otimes \rho \otimes \mu}(n)
\sum_{\bm{\eps}\in \set{\pm 1}^K} N^{\bm{\eps}}_\rho(\bm{x};i)
\prod_{1\le i\le K} x_u^{(1-\eps_u) c[n]/2}.
\end{equation*}
If $c[n]$ is large enough in terms of $\bm{d}$,
the $\prod_{1\le u\le K} x_i^{(1-e_u) c[n]/2 + e_u d_u} \prod_{K+1\le u\le K+Q} y_{u-K}^{d_u}$ coefficient of the last display equals the $\prod_{1\le u\le K} x_i^{e_u d_u} \prod_{K+1\le u\le K+Q} y_{u-K}^{d_u}$ coefficient of
\begin{equation*}
\sum_{0\le i\le 1}
\sum_{\substack{\rho,\mu\in \ptn: \\ \rho_1\le K,\; l(\mu)\le Q, \\ \abs{d_1+\dots+d_K - \abs{\rho}} \le C_K, \\ \abs{\mu} \le d_{K+1}+\dots+d_{K+Q}}}
m^1_\mu(\bm{y})
\tr_{\det^i \otimes \rho \otimes \mu}(n)
N^{\bm{e}}_\rho(\bm{x};i),
\end{equation*}
in light of the degree properties of $m^1_\mu(\bm{y})$ and $N^{\bm{\eps}}_\rho(\bm{x};i)$ mentioned after Definition~\ref{DEFN:stable-numerators}.
Taking $n\to \infty$, and recalling the definition of $T^i_{\rho \otimes \mu}$ from Lemma~\ref{LEM:stable-interchange}, we deduce that this coefficient converges precisely to the $\prod_{1\le u\le K} x_i^{e_u d_u} \prod_{K+1\le u\le K+Q} y_{u-K}^{d_u}$ coefficient of $\mathsf{SN}(\bm{e},\bm{x},\bm{y})$.
Since $\bm{d}$ was arbitrary, (1) follows.


(2):
Let $\Cavg^\ast(\bm{\eps},\bm{d}) \defeq \lim_{n\to \infty} \Cavg^\ast(\bm{\eps},\bm{d};n)$.
Since $F^\ast(\bm{e},\bm{x},\bm{y};n) \in \CC[[x_u^{\eps_u}, \bm{y}]]$
for all $n$, we have $\Cavg^\ast(\bm{\eps},\bm{d}) = 0$ unless $d_1,\dots,d_{K+Q} \ge 0$.
Suppose for contradiction that there exists $\bm{d}\in \ZZ_{\ge 0}^{K+Q}$ such that $\Cavg^\ast(\bm{\eps},\bm{d})$ does not exist.
Assume $\bm{d}$ to be minimal with respect to the lexicographic order on $\ZZ^{K+Q}$.
Let $x_1^{\eps_1 l_1} \cdots x_K^{\eps_K l_K}$ be the lowest degree monomial appearing in $D_K(\bm{x})$, with respect to the lexicographic order on tuples $\bm{l} = (l_1,\dots,l_K) \in \ZZ^K$.
Then the $x_1^{\eps_1 (l_1+d_1)} \cdots x_K^{\eps_K (l_K+d_K)} y_1^{d_{K+1}} \cdots y_Q^{d_{K+Q}}$ coefficient of $D_K(\bm{x}) F^\ast(\bm{e},\bm{x},\bm{y};n)$
is a fixed $\CC$-linear combination of $\Cavg^\ast(\bm{\eps},\bm{d};n)$ (with a nonzero coefficient) and possibly some other terms $\Cavg^\ast(\bm{\eps},\bm{d}';n)$ with $\bm{d}' < \bm{d}$.
By (1) and the minimality assumption on $\bm{d}$, it follows that $\Cavg^\ast(\bm{\eps},\bm{d})$ exists,
which contradicts our assumptions.

(3):
Axiom~A is immediate from (2) and \eqref{EQN:approx-FE}.
In fact, under the substitution \eqref{EQN:xy-to-s-substitution},
it is clear that (2) implies $\Cavg(\bm{\eps},\bm{d}) = (-1)^{d_1+\dots+d_K} \Cavg^\ast(\bm{\eps},\bm{d})$,
whence $\lim_{n\to \infty} F^\ast(\bm{e},\bm{x},\bm{y};n) = \mathsf{F}(\bm{e},\bm{s})$ (see \eqref{EQN:define-F-series}).
Axiom~F follows from Axiom~A, since $\Cfam = \Cavg((1,\dots,1),\bm{0})$.

(4):
$\lim_{n\to \infty} D_K(\bm{x}) F^\ast(\bm{e},\bm{x},\bm{y};n)
= D_K(\bm{x}) \lim_{n\to \infty} F^\ast(\bm{e},\bm{x},\bm{y};n)$
by (2), since $D_K$ has finitely many terms.
Now (1) and (3) give the desired result.
\end{proof}

\begin{lemma}
[Polarized convergence]
\label{LEM:stable-polarized-bound}


Let $\bm{\eps}\in \set{\pm 1}^K$, $\bm{x}\in \CC^K$, and $\bm{y}\in \CC^Q$.
Assume
\begin{equation}
\label{INEQ:open-region-try-1}
q^{-\beta} < \abs{x_u} < q^\beta,
\quad \abs{y_v} < q^{-\gamma}/\Cbetti',
\quad q^{1/6} \min(1,\gamma)^{\Cstabtwo} \ge 2 (2 q^\beta \Cbetti)^{\Cstabtwo},
\end{equation}
where $\beta,\gamma>0$.
Then $\mathsf{SN}(\bm{\eps},\bm{x},\bm{y})$ converges absolutely uniformly
to a holomorphic function.
\end{lemma}

\begin{proof}
By Lemma~\ref{LEM:stable-interchange}, the Weierstrass M-test, and Morera's theorem, it suffices to shows that the (countable) series
\begin{equation*}
\sum_{0\le i\le 1}
\sum_{\substack{\rho,\mu\in \ptn: \\ \rho_1\le K,\; l(\mu)\le Q}}
\sum_{k\ge 0}
\abs{N^{\bm{\eps}}_\rho(\bm{x};i)
m^1_\mu(\bm{y})
T^i_{\rho \otimes \mu}(k)}
\end{equation*}
converges uniformly\footnote{with respect to a fixed ordering of the terms in the sum} over $\bm{x}$ and $\bm{y}$.
For this, we mimic the strategy of Lemma~\ref{LEM:stable-gap-bounds}.
Any Laurent polynomial $N^{\bm{\eps}}_\rho(\bm{x};i) \in \CC[[x_u^{\eps_u}, \bm{y}]][x_u^{-\eps_u},\bm{y}^{-1}]$ has at most $4^{K(K-1)/2} 2^K K! = 2^{K^2} K!$ terms, each of which has multi-degree in $\prod_{1\le u\le K} [\rho^\prime_u - 10K^2, \rho^\prime_u + 10K^2]$ up to permutation.
Thus $\abs{N^{\bm{\eps}}_\rho(\bm{x};i)}
\le 2^{K^2} K! (q^\beta)^{\abs{\rho} + 10K^3}$.
Therefore, the last display is
\begin{equation*}
\begin{split}
&\ll_K \sum_{\substack{k\ge 0, \\ \rho\belongs (K^{c[n_k]/2 + Q}), \\ l(\mu)\le Q}}
(q^\beta)^{K (c[n_k]/2 + Q) + 10K^3}
m^1_\mu(q^{-\gamma},\dots,q^{-\gamma})
q^{-k/2} \Cbetti^{d(n_{k,2})} \chi^G_\rho \chi^G_\mu(I_{c[n_{k,2}]}) \\
&\ll_{q,\beta,K,Q} \sum_{k\ge 0}
q^{\beta K c[n_k]/2}
q^{-k/2} \Cbetti^{d(n_{k,2})}
\frac{2^{c[n_{k,2}] K} 2^{c[n_{k,2}] Q}}{\min(1,\gamma)^{c[n_{k,2}] Q}}
\le \sum_{k\ge 0} q^{-k/2}
\frac{(2 q^\beta \Cbetti)^{\Cstabtwo k + \Cstabthr}}
{\min(1,\gamma)^{\Cstabtwo k + \Cstabthr}},
\end{split}
\end{equation*}
with the final geometric series converging as in the proof of Lemma~\ref{LEM:stable-gap-bounds}.
\end{proof}

\begin{lemma}
[Polarization identity]
\label{LEM:polarization-identity}
Let $(\bm{x},\bm{y}),\bm{s}\in \CC^K \times \CC^Q$
with \eqref{EQN:xy-to-s-substitution}.
Let $\beta,\gamma>0$ with \eqref{INEQ:open-region-try-1}.
Then $\mathsf{F}(\bm{\eps},\bm{s})$ from \eqref{EQN:define-F-series} is meromorphic,
$\mathsf{MT}(\bm{s};n)$ from \eqref{EQN:define-MT-series} is holomorphic,
and
\begin{equation*}
\mathsf{ST}(\bm{x},\bm{y};n) = \mathsf{MT}(\bm{s};n).
\end{equation*}
\end{lemma}

\begin{proof}





By Lemma~\ref{LEM:stable-gap-bounds} (and its proof),
the Weierstrass M-test,
and Morera's theorem,
$\mathsf{ST}(\bm{x},\bm{y};n)$
converges absolutely uniformly to a holomorphic function.
On the other hand, by \eqref{EQN:expand-stable-numerator} and Lemma~\ref{LEM:stable-polarized-bound}, the following identity of holomorphic functions holds:
\begin{equation*}
\mathsf{ST}(\bm{x},\bm{y};n) D_K(\bm{x})
= \sum_{\bm{\eps}\in \set{\pm 1}^K} \mathsf{SN}(\bm{\eps},\bm{x},\bm{y}) \prod_{1\le i\le K} x_i^{(1-\eps_i) c[n]/2}.
\end{equation*}
Lemma~\ref{LEM:stable-polarized-bound} actually also implies that the power series $(z_1^{\eps_1}\cdots z_K^{\eps_K})^{10K^2} \mathsf{SN}(\bm{\eps},\bm{z},\bm{y}) \in \CC[[z_u^{\eps_u}, \bm{y}]]$ is absolutely convergent for $\abs{z_u^{\eps_u}}<q^\beta$ and $\abs{y_v}<q^{-\gamma}$.
By Lemma~\ref{LEM:key-formal-coefficient-calculation}(4) and the factorization of $D_K$,
the ratio $\mathsf{SN}(\bm{\eps},\bm{z},\bm{y})/D_K(\bm{z})$ is therefore holomorphic for $0<\abs{z_u^{\eps_u}}<1$ and $\abs{y_v}<q^{-\gamma}$.
The series expansion of $\mathsf{SN}(\bm{e},\bm{x},\bm{y})/D_K(\bm{x})$ must then coincide with the formal series $\mathsf{F}(\bm{e},\bm{s})$,
since $\CC[[z_u^{\eps_u}, \bm{y}]]$ is an integral domain.
This makes $\mathsf{F}(\bm{e},\bm{s})$ meromorphic.
Dividing the last display by $D_K$
and plugging in \eqref{EQN:xy-to-s-substitution}
concludes the proof, by \eqref{EQN:define-MT-series}.
\end{proof}

\begin{proof}
[Proof of Theorem~\ref{THM:general-axiomatization}]

Let $(\bm{x},\bm{y}),\bm{s}\in \CC^K \times \CC^Q$
with \eqref{EQN:xy-to-s-substitution}.
Assume \eqref{INEQ:convenient-q-size-assumption} and \eqref{INEQ:q-restricted-ratios-region}.
Let $\beta = \delta > 0$ and $\gamma = q^{-\delta} < 1$.
Since $q\ge 2^{12} (2\Cbetti)^{1/(2\delta)}$, we have
\begin{equation*}
q\ge 2^6 (2\Cbetti)^{1/(4\delta)} q^{1/2}
= 2^6 (2q^{2\delta}\Cbetti)^{1/(4\delta)}
= 2^6 (2q^\beta\Cbetti/\gamma)^{\max(6\Cstabtwo, 42(K+Q)\Crank^3\Cslope)},
\end{equation*}
where in the final step we plug the definition of $\delta$ into the exponent $1/(4\delta)$.
In particular, \eqref{INEQ:q-conditions-try-1} and \eqref{INEQ:open-region-try-1} both hold.
Plugging \eqref{EQN:expand-ratio-average} and Lemma~\ref{LEM:polarization-identity}
into Lemma~\ref{LEM:trace-formulation-of-main-theorem},
we therefore deduce that the left-hand side of \eqref{INEQ:main-general-goal} is
\begin{equation*}
\le 5 q^{\Cstabfou}
(2 q^\beta \Cbetti/\gamma)^{\Crank d(n) (K+Q)}
q^{-d(n)/(6\Crank^2\Cslope)}
\le 5 q^{\Cstabfou}
(q/2^6)^{d(n)/(42\Crank^2\Cslope)}
q^{-d(n)/(6\Crank^2\Cslope)},
\end{equation*}
since $(2q^\beta\Cbetti/\gamma)^{42(K+Q)\Crank^3\Cslope} \le q/2^6$.
This suffices for \eqref{INEQ:main-general-goal}, since $\frac16 - \frac1{42} = \frac{1}{7}$.
Furthermore, the desired analytic properties of $\mathsf{F}(\bm{\eps},\bm{s})$ and $\mathsf{MT}(\bm{s};n)$ follow from Lemma~\ref{LEM:polarization-identity}.
\end{proof}

\section{The quadratic case}
\label{SEC:quadratic-analysis}


We now specialize to the classical quadratic Dirichlet $L$-function setting introduced in \S\ref{SEC:intro}.
For $n\ge 1$, define the \emph{unordered configuration space}
\begin{equation*}
\map{UConf}_n(\Aff^1) \defeq
\set{(a_1,\dots,a_n): \disc(t^n+a_1t^{n-1}+\dots+a_n)\ne 0} \belongs \Aff^n_{\FF_q}.
\end{equation*}
Then $\map{UConf}_n(\Aff^1)(\FF_q) = \mscr{P}_n$.
Now let $q$ be odd.
Let $p_n\maps \mcal{C}_n\to \map{UConf}_n(\Aff^1)$ be the universal affine hyperelliptic curve $$y^2 = t^n+a_1t^{n-1}+\dots+a_n$$ (punctured at $t=\infty$);
let $\ol{p}_n\maps \ol{\mcal{C}}_n \to \map{UConf}_n(\Aff^1)$ be the usual smooth proper model thereof.
Let (cf.~\cite{bergstrom2023hyperelliptic}*{\S8.1.4 and \S9.1.3})
\begin{equation*}
\VV_n\defeq R^1(p_n)_!\ol{\QQ}_\ell(1/2),
\quad \VV^0_n\defeq R^1(\ol{p}_n)_!\ol{\QQ}_\ell(1/2),
\end{equation*}
say with $\ell=2\nmid q$.
By Poincar\'{e} duality, $\VV^0_n$ is a pure, weight $0$, symplectic local system on $\map{UConf}_n(\Aff^1)$ of rank $2\floor{(n-1)/2} \in \set{n-2, n-1}$.
There is an \emph{excision exact sequence}
\begin{equation*}
0 \to \ol{\QQ}_\ell(1/2)^{n-1-2\floor{(n-1)/2}} \to \VV_n \to \VV^0_n \to 0,
\end{equation*}
as recorded in \cite{bergstrom2023hyperelliptic}*{\S8.1.4}, because $\mcal{C}_n$ has $n-2\floor{(n-1)/2} \in \set{1,2}$ punctures.

For simplicity, we focus on $n=2g+1$ odd, in which case $\VV_n = \VV^0_n$.
A similar analysis may be possible for $n=2g+2$ even, after accounting for the difference between $\VV_n$ (which is mixed) and $\VV^0_n$.
The purity assumption in Axiom~L prevents us from handling both parities simultaneously,
but it lets us keep \S\ref{SEC:general-framework} simple
(as we prefer for the present note).

\subsection{Proof of Theorem~\ref{THM:apply-stability-1}}

Fix integers $K,Q\ge 0$.
We now put the local systems $\VV_{2g+1}$ into the framework of \S\ref{SEC:general-framework}, indexing by $g\ge Q$ for convenience (instead of by $n\ge 1$ as in \S\ref{SEC:general-framework}).
For $g\ge Q$, let
\begin{equation*}
\mcal{P}_g \defeq \map{UConf}_{2g+1}(\Aff^1),
\quad V_g \defeq \VV_{2g+1}.
\end{equation*}
Then $\mcal{P}_g$ has dimension $d(g) = 2g+1$, and $V_g$ has rank $c[g] = 2g$.
Thus Axiom~L holds with $G=\Sp$ and $\Crank=1$.
Axiom~B holds with $(\Cbetti,\Cbetti') = (2,1)$, by \cite{bergstrom2023hyperelliptic}*{\S1.3.9 and Lemma~11.3.13}.
Axiom~S, with $(\Cslope,\Cstart) = (12,13)$, holds by \cite{MPPRW}*{Theorem~1.4}.

Implicitly, in justifying Axioms~B and~S, we rely on certain \emph{comparison isomorphisms} and \emph{stabilization maps} from \cite{bergstrom2023hyperelliptic}.
The former come from careful use of Artin's comparison theorem as in \cite{bergstrom2023hyperelliptic}*{\S\S8.1.3--8.1.4}
or \cite{ellenberg2016homological}*{proof of Proposition~7.7}.
The latter, i.e.~the stabilization maps, are $\Gal(\ol{\FF}_q/\FF_q)$-equivariant maps
\begin{equation}\label{MAP:increment-log-conductor-g}
H_k(\mcal{P}_g \otimes \ol{\FF}_q, \Schur^G_\lambda(V_g))
\to H_k(\mcal{P}_{g+1} \otimes \ol{\FF}_q, \Schur^G_\lambda(V_{g+1}))
\end{equation}
for $k\ge 0$ and $\lambda\in \ptn$, defined using log geometry
following \cite{bergstrom2023hyperelliptic}*{\S1.3.12, \S8.5, \S9.1.6}.
Here are the key points underlying \eqref{MAP:increment-log-conductor-g}:
\begin{enumerate}
    \item There is no obvious morphism $\mcal{P}_g \to \mcal{P}_{g+1}$,
    but after passing to suitable (Abramovich--Corti--Vistoli) modular compactifications there is a satisfactory \emph{log morphism}.
    \item The local systems $\Schur^G_\lambda(V_g)$ for $g\ge Q$ then fit into a \emph{coefficient system} $V^\infty_\lambda$ by \cite{bergstrom2023hyperelliptic}*{\S3.4.4, \S8.6.4, \S9.1.1}, via certain \emph{gluing} and \emph{collapse} maps.
\end{enumerate}
One can also define \eqref{MAP:increment-log-conductor-g} topologically, but it is essential for us that \eqref{MAP:increment-log-conductor-g} is compatible with the action of the geometric Frobenius element $\Fr_q\in \Gal(\ol{\FF}_q/\FF_q)$.
Of course, in many cases, both sides of \eqref{MAP:increment-log-conductor-g} actually vanish, and then $\Fr_q$-equivariance holds trivially.

In any case, with Axioms~L, B, and~S in hand, we may now apply Theorem~\ref{THM:general-axiomatization}.
By \eqref{EQN:define-Cstab1}, \eqref{EQN:define-Cstab2}, and \eqref{INEQ:convenient-q-size-assumption},
we have $\Cstabone = 2 \max(1, K+Q)$, $\Cstabtwo = K \Cslope + \Cstabone \Cslope$, and
\begin{equation}
\label{EQN:quadratic-specialized-delta-value}
\delta = \frac{1}{24 \Cslope \max(K + \Cstabone, 7(K+Q))}
= \frac{1}{288 \max(2, 7(K+Q))},
\quad \frac{d(g)}{7\Crank^2\Cslope} = \frac{2g+1}{84}.
\end{equation}
If $q \ge 2^{12} 2^{1/\delta}$, then $q \ge 2^{12} (2\Cbetti)^{1/(2\delta)}$, since $\Cbetti = 2$.
Now, since $c[g] \ge 2Q$ for all $g\ge Q$,
Theorem~\ref{THM:apply-stability-1} for $n=2g+1\ge 2Q+1$
follows immediately from
Theorem~\ref{THM:general-axiomatization} for $g\ge Q$.


\subsection{Proof of Corollary~\ref{COR:integrate-stability-1}}


Let $K=0$ and $Q=1$ in Theorem~\ref{THM:apply-stability-1}.
Then $\mathsf{RR}_L(s;n) = \card{\mscr{P}_n}$ by \cite{florea2021negative}*{(59)},
so
\begin{equation*}
\frac{1}{\card{\mscr{P}_n}} \sum_{d\in \mscr{P}_n} \frac{1}{L(s,\chi_d)}
= 1
+ O(q^{O_{K,Q}(1)} q^{-\omega n})
\end{equation*}
for $\Re(s) > \frac12 + q^{-\delta}$.
For $q$ large enough in terms of $I$,
we have $2q^{-\delta} \max(I) \le \omega/2$.
Now multiply the last display by $(q^R)^{s - 1/2}$,
and average over the line $\Re(s) = \frac12+2q^{-\delta}$.

\subsection{Proof of Theorem~\ref{THM:one-level-density}}

We will roughly follow \cite{andrade2014conjectures}*{\S7} (cf.~\cite{conrey2007applications}*{\S3}),
but it would be interesting to know if a more direct approach is possible (which might then extend more easily to $K$-level correlations for $K\ge 2$).
Let $K=Q=1$ and $\delta^\ast_q \defeq 2q^{-\delta}$.
Assume
$2\nmid q\ge 2^{12} 2^{1/\delta}$
and $\delta^\ast_q
\le 0.4\delta$,
by taking $q\gg_\delta 1$.
Write $n=2g+1$, where $g\ge 1$, and let
\begin{equation*}
R(s_1,s_2;n)
\defeq \sum_{d\in \mscr{P}_n} \frac{L(s_1,\chi_d)}{L(s_2,\chi_d)}.
\end{equation*}
Cauchy's integral formula,
followed by \eqref{INEQ:main-quadratic-goal} for $K=Q=1$,
tells us that if $$\Re(s_1) = \Re(s_2) = \frac12 + \delta^\ast_q,$$ then
\begin{equation*}
\begin{split}
\frac{\partial (R(s_1,s_2;n) - \mathsf{RR}_L(s_1,s_2;n))}{\partial s_1}
&= \frac{\int_{(\delta/2)} - \int_{(-\delta/2)}}{2\pi i} \frac{R(s_1+z,s_2;n) - \mathsf{RR}_L(s_1+z,s_2;n)}{z^2} \, dz \\
&\ll \frac{1}{\log{q}} \frac{q^{O(1)} q^{(1-\omega)n}}{\delta^2},
\end{split}
\end{equation*}
where as in \cite{andrade2014conjectures}*{\S7} we let $(c)$ denote the vertical path from $c - \frac{\pi i}{\log{q}}$ to $c + \frac{\pi i}{\log{q}}$.

For simplicity, write $A^1(s,s) \defeq \frac{\partial A(s_1,s_2)}{\partial s_1} \vert_{s_1=s_2=s}$
for any function $A$ of $(s_1,s_2)$.
Let $F(x;n) \defeq \sum_{k\in \ZZ} f(\frac{\log(q^{n-1})}{2\pi} (x + \frac{2\pi k}{\log{q}}))$ be the natural normalized $q$-adic periodization of $f$, as in \cite{rudnick2008traces}*{\S1.1}.
By Cauchy's residue theorem, the left-hand side of \eqref{1-level-density-convergence} equals
\begin{equation}
\label{residue-integral}
\frac{\int_{(1/2 + \delta^\ast_q)} - \int_{(1/2 - \delta^\ast_q)}}{2\pi i}
\frac{F((s-1/2)/i;n)}{\card{\mscr{P}_n}} R^1(s,s;n)\, ds.
\end{equation}
But $L(s,\chi_d) = q^{(n-1)(1/2 - s)} L(1-s,\chi_d)$ by \eqref{EQN:FE} (since $2\nmid n$),
so $$L'(s,\chi_d)/L(s,\chi_d) = -\log(q^{n-1}) - L'(1-s,\chi_d)/L(1-s,\chi_d).$$
Furthermore, $F(x;n) = F(-x;n)$,
since $g$ and $f$ are even.
Thus \eqref{residue-integral} equals
\begin{equation*}
\frac{\int_{(1/2 - \delta^\ast_q)}}{2\pi i} F((s-1/2)/i;n)
\log(q^{n-1})\, ds
+ \frac{\int_{(1/2 + \delta^\ast_q)}}{\pi i}
\frac{F((s-1/2)/i;n)}{\card{\mscr{P}_n}} R^1(s,s;n)\, ds.
\end{equation*}
The first term equals $\log(q^{n-1}) \frac{\int_\RR}{2\pi} f(\frac{\log(q^{n-1})}{2\pi} x)\, dx = \int_\RR f(x)\, dx$.
The second term is
\begin{equation*}
\frac{\int_{(1/2 + \delta^\ast_q)}}{\pi i}
\frac{F((s-1/2)/i;n)}{\card{\mscr{P}_n}} \mathsf{RR}^1_L(s,s;n)\, ds
+ O{\left(\frac{\sup_{x\in \RR} \abs{F(x - i\delta^\ast_q;n)}}{\card{\mscr{P}_n} \log{q}}
\frac{1}{\log{q}} \frac{q^{O(1)} q^{(1-\omega)n}}{\delta^2}\right)}.
\end{equation*}

Suppose $I\belongs [-\lambda,\lambda]$.
For $x\in \CC$, we have $\abs{f(x)} \le e^{\abs{2\pi\lambda\Im(x)}} \int_{\xi\in \RR} \abs{g(\xi)}\, d\xi$ trivially,
and $\abs{f(x)} \le \abs{2\pi x}^{-2} e^{\abs{2\pi\lambda\Im(x)}} \int_{\xi\in \RR} \abs{g''(\xi)}\, d\xi$ by integration by parts in $\xi$.
So if $x\in \RR$, then
\begin{equation*}
\begin{split}
\abs{F(x - i\delta^\ast_q;n)}
&\ll_g \sum_{k\in \ZZ} \frac{e^{\lambda \log(q^{n-1}) \delta^\ast_q}}
{\max(1, \abs{\log(q^{n-1})\, (x + \frac{2\pi k}{\log{q}})}^2)} \\
&\le q^{(n-1) \lambda \delta^\ast_q}
\left(1 + \sum_{k\ne 0} \frac{1}{\abs{\log(q^{n-1})\, (\frac{\pi k}{\log{q}})}^2}\right)
\ll q^{(n-1) \lambda \delta^\ast_q}
\le q^{2n\lambda q^{-\delta}}.
\end{split}
\end{equation*}
Therefore, the penultimate display is
\begin{equation*}
\frac{\int_{(1/2 + \delta^\ast_q)}}{\pi i}
\frac{F((s-1/2)/i;n)}{\card{\mscr{P}_n}} \mathsf{RR}^1_L(s,s;n)\, ds
+ o_{g,\delta,q; n\to \infty}(1),
\end{equation*}
assuming $2\lambda q^{-\delta} < \omega$ (as we may do by taking $q\gg_\lambda 1$).

To complete our analysis of the residue integral \eqref{residue-integral},
it remains only to show that
\begin{equation*}
\frac{\int_{(1/2 + \delta^\ast_q)}}{\pi i}
\frac{F((s-1/2)/i;n)}{\card{\mscr{P}_n}} \mathsf{RR}^1_L(s,s;n)\, ds
\to - \int_\RR \frac{\sin(2\pi x)}{2\pi x} f(x)\, dx
\end{equation*}
as $n\to \infty$.
This is already done in some form in \cite{andrade2014conjectures}*{\S7}, without any $q$-restriction on the size of $\lambda$.
Specifically, they show that the left-hand side equals
\begin{equation*}
\frac{\int_{-i\infty}^{i\infty}}{\pi i}
f{\left(\frac{\log(q^{n-1}) s}{2\pi i}\right)}
\left(\frac{\zeta'_{\FF_q[t]}(1+2s)}{\zeta_{\FF_q[t]}(1+2s)}
+ A^1_D(s,s)
- q^{-(n-1)s} (\log{q})\zeta_{\FF_q[t]}(1-2s) A_D(-s,s) \right)\, ds,
\end{equation*}
where $A_D(\alpha,\gamma)$ is an explicit, absolutely convergent Euler product for $\Re(\alpha),\Re(\gamma) > -\frac14$,
and where $A_D(\alpha,\alpha) = 1$.
Here
\begin{equation*}
\begin{split}
\zeta_{\FF_q[t]}(1-2s) = 1/(1-q^{2s}) &= (-2s\log{q})^{-1} + O(1), \\
\zeta'_{\FF_q[t]}(1+2s)/\zeta_{\FF_q[t]}(1+2s) &= (-2s)^{-1} + O(1), \\
\end{split}
\end{equation*}
for $\abs{s} \le \pi/(10\log{q})$,
so up to $o_{f,q; n\to \infty}(1)$, the penultimate display is
\begin{equation*}
\frac{\int_{-i\pi/(10\log{q})}^{i\pi/(10\log{q})}}{\pi i}
f{\left(\frac{\log(q^{n-1}) s}{2\pi i}\right)}
\frac{(1
- q^{-(n-1)s} (\log{q})(\log{q})^{-1})\, ds}{-2s},
\end{equation*}
which in turn simplifies, up to $o_{f,q; n\to \infty}(1)$, to
\begin{equation*}
\frac{\int_{-i\infty}^{i\infty}}{\pi i}
f{\left(\frac{\log(q^{n-1}) s}{2\pi i}\right)}
(q^{-(n-1)s}-1)\, \frac{ds}{2s}
= \frac{\int_{-\infty}^{\infty}}{\pi i}
f(x)
(e^{-2\pi ix}-1)\, \frac{dx}{2x}.
\end{equation*}
This suffices,
because (1) $e^{-2\pi ix} = \cos(2\pi x) - i\sin(2\pi x)$ and (2) $f$ and $\cos$ are even.

\appendix

\section{Discussion of geometric families}

There are many interesting families aside from the hyperelliptic ensemble featured in \S\ref{SEC:quadratic-analysis}.
For example, one can ask about moduli spaces of curves, abelian varieties, hypersurfaces, complete intersections, or fibered varieties thereof.
Some of these are known to be quite difficult.
For instance, $\mcal{M}_g$ (the moduli space of genus $g$ curves) fails Axiom~B,
and in fact Axiom~F remains open; see \cites{lipnowski2018large,chan2021tropical}, and references therein, for details.
Even though $\mcal{M}_g$ satisfies Axiom~S after decoration
\cite{bergstrom2023hyperelliptic}*{\S1.3.9\cite{MPPRW}*{Theorem~1.1}},
we really also need Axiom~B in order for the present proof of Theorem~\ref{THM:general-axiomatization} to work.

An interesting class of families are the (function-field analogs of the) \emph{Sarnak--Shin--Templier geometric families} put forth in \cite{sarnak2016families}.
Roughly speaking, these are \emph{linear} parametric families of $L$-functions of varieties fibered over $\Aff^1_{\FF_q}$.
For these, Axioms~F and~A are often already provable in practice, by the square-free sieve \cite{poonen2003squarefree} and the Deligne--Katz equidistribution theorem.
Therefore, there is hope that Axioms~B and~S could be true in many interesting cases.
Here is a list of some possible examples to consider:
\begin{enumerate}
    \item Families of $L(s,\chi_d)$ or $\chi_d$ twists with congruence conditions, like $d_0\mid d$ or $d\equiv 1\bmod{d_0}$.
    This might be useful for sieving purposes, or for filtering out root numbers.

    \item Quadratic and cubic twist families of elliptic curve $L$-functions, as considered in e.g.~\cite{meisner2022low}.
    Other families of elliptic curve $L$-functions could also be interesting.

    \item Branched covers of toric varieties.
    For instance, one could study $L$-functions associated to equations of the form $y^k = D(c_1,\dots,c_n)$, where $k$ and $D$ are fixed and $c_1,\dots,c_n\in \FF_q[t]$ are varying.
    The ``discriminant'' $D$ could essentially be any polynomial.

    \item Examples from the circle method,
    as discussed in \cite{BGW2024forthcoming}.
\end{enumerate}











A good strategy so far for proving Axiom~B has been to consider ``root covering'' maps
such as $\map{Conf}_n(\Aff^1) \to \map{UConf}_n(\Aff^1)$.
Roughly speaking, this map pulls back the irreducible (discriminant) boundary downstairs to a highly symmetric, reducible boundary upstairs, which one may then hope is simpler to study.
In addition to \cite{bergstrom2023hyperelliptic}*{proof of Lemma~11.3.13},
see \cite{sawin2020representation}*{proof of Lemma~2.11} in the case of moments,
or \cite{sawin2022square} for certain ratios (possibly all ratios without conjugates),
in certain harmonic families.
Sawin has also suggested to us that characteristic cycles as in \cite{sawin2022square} could be helpful in general,
or even essential in harmonic families at least.

For Axiom~S, one would like to generalize the strategy of \cite{MPPRW},
which is based heavily on a precise understanding of near-surjective monodromy representations and their images.
One would hope for classical topological or geometric methods to shed light here.

An important issue we have glossed over is Axiom~L;
in practice, one often has to break up an initial family into smaller pieces on which the $L$-functions can be described by a local system (and not just a constructible sheaf).
A general strategy for this is to use the square-free sieve \cite{poonen2003squarefree} and the Krasner-type lemma of \cite{kisin1999local} to reduce to considering families where the discriminant is square-free away from a finite set $S$.
On $S$ one can often fix local $L$-factors using \cite{kisin1999local},
and away from $S$ one can often use tools like Proposition~\ref{PROP:square-free-case} to explicitly compute all $L$-factors in a concrete, uniform geometric manner.

\begin{proposition}
\label{PROP:square-free-case}
Let $K$ be a non-archimedean local field,
with ring of integers $\OK$
and residue field $k \cong \FF_q$.
Let $X$ be a proper $\OK$-scheme,
whose generic fiber $X_K$ is smooth of dimension $n\ge 1$.
Then each of the following implies the next:
\begin{enumerate}
\item $X$ is a hypersurface,
$\map{char}(\FF_q) \ne 2$,
and the discriminant of $X$ has valuation $\le 1$.

\item Either $X$ is smooth over $\mcal{O}$, or (i)--(iii) hold:
(i) $X$ is regular,
(ii) $X_k$ has a unique singular point $x_0$,
and (iii) $x_0$ is a non-degenerate double point of $X_k$.
Furthermore, $H^i(X_K)$ has pure monodromy filtration for every integer $i\ge 0$.
In addition, for all $i\ge 0$ with $i\ne n$, we have $H^i(X_K) \cong H^i(\PP^n_K)$ and $H^i(X_k) \cong H^i(\PP^n_k)$ in $\ell$-adic cohomology for every prime $\ell\nmid q$.

\item 
$\det(1 - T \Fr_q, H^i(X_K)^I)
= \det(1 - T \Fr_q, H^i(X_k))$
for all $i\ge 0$.
\end{enumerate}
\end{proposition}







\begin{proof}
(1)$\Rightarrow$(2):
For the first sentence, use \cite{poonen2020valuation}*{Theorem 1.1}.
For the second, use \cite{scholze2012perfectoid},
or earlier work of Deligne if $\map{char}(K)>0$;
see \cite{kahn2020zeta}*{\S5.6.4} for details.
For the third, use \cite{lindner2020hypersurfaces}*{Theorem~1.2} (which assumes $\map{char}(\FF_q) \ne 2$)
and \cite{wang2023dichotomous}*{Theorem~2.13 and Remark~2.14} (due to Deligne, Katz, Skorobogatov, and Ghorpade--Lachaud).

(2)$\Rightarrow$(3):
If $X$ is smooth, use smooth proper base change.
Now assume (i)--(iii) instead.
We follow the approach suggested by \cite{gortz2004computing}*{p.~155} and \cite{rapoport1990bad}*{p.~267--268}.\footnote{If $n=1$, then \cite{bouw2017computing}*{Proposition~2.8} offers another route
(since $X$ is flat by miracle flatness, and is therefore a regular model of $X_K$ in the sense of \cite{bouw2017computing}*{\S1.4}).}
Let $R^u\Psi\QQ_\ell$ be the $u$th \emph{sheaf of nearby cycles} on $X_{\ol{k}}$, as in \cite{gortz2004computing}*{\S4.1}.
Let $\tr^{ss}(\Fr_q^j, H^i(X_K))$ be the \emph{semi-simple trace} \cite{rapoport1990bad}*{(2.11)} of $\Fr_q^j$ with respect to the monodromy filtration on $H^i(X_K)$.
Let $\pi\maps X'\to X$ be the blow-up of $X$ at $x_0$;
then $X'$ has semi-stable reduction in the sense of \cite{gortz2004computing}*{p.~155}.
Using the weight spectral sequence,
the trivial inertia action of \cite{gortz2004computing}*{Theorem~4.1},
and the Grothendieck--Lefschetz trace formula,
we get
\begin{equation*}
\begin{split}
\sum_{i\ge 0} (-1)^i \tr^{ss}(\Fr_q^j, H^i(X'_K))
&= \sum_{u,v\ge 0} (-1)^{u+v} \tr(\Fr_q^j, H^v(X'_{\ol{k}},R^u\Psi'\QQ_\ell)^I) \\
&= \sum_{u,v\ge 0} (-1)^{u+v} \tr(\Fr_q^j, H^v(X'_{\ol{k}},(R^u\Psi'\QQ_\ell)^I)) \\
&= \sum_{x\in X'(\FF_{q^j})} \tr(\Fr_q^j, (R\Psi'\QQ_\ell)_x^I)
= \sum_{x\in X'(\FF_{q^j})} \tr(\Fr_q^j, (R\Psi'\QQ_\ell)_x),
\end{split}
\end{equation*}
where $\tr(\Fr_q^j, (R\Psi\QQ_\ell)_x^I)$ and $\tr(\Fr_q^j, (R\Psi'\QQ_\ell)_x)$ denote alternating traces as on \cite{gortz2004computing}*{p.~155}.
Here $X'_K = X_K$, so $\tr^{ss}(\Fr_q^j, H(X_K)) = \tr^{ss}(\Fr_q^j, H(X'_K))$.

Let $E\belongs X'$ be the \emph{exceptional divisor} of $\pi$;
then $E\cong \PP^n_k$.
Let $Y\belongs X'$ be the \emph{strict transform} of $X_k$,
i.e.~the closure of $X_k\setminus \set{x_0}$ in $X'$.
Let $Z\defeq Y\cap E$;
then $Z$ is a smooth quadric over $k$ of dimension $n-1$,
so for some $\alpha_Z \in \set{\pm q^{(n-1)/2}}$, we have
$$\card{Z(\FF_{q^j})} = \card{\PP^{n-1}(\FF_{q^j})} + \alpha_Z^j \bm{1}_{2\mid n-1}.$$
The formulas in \cite{gortz2004computing}*{Theorem~4.1}
give $\tr(\Fr_q^j, (R\Psi'\QQ_\ell)_x) = 1 - q^j \bm{1}_{x\in Z}$ for $x\in X'(\FF_{q^j})$, so
\begin{equation*}
\sum_{x\in X'(\FF_{q^j})} \tr(\Fr_q^j, (R\Psi'\QQ_\ell)_x)
= \card{X'(\FF_{q^j})} - q^j \card{Z(\FF_{q^j})}
= \card{X(\FF_{q^j})}-1 + \card{E(\FF_{q^j})} - q^j \card{Z(\FF_{q^j})},
\end{equation*}
since $E\belongs X'_k$ and $X'_k\setminus E = X_k\setminus \set{x_0}$.
All in all, we get
\begin{equation*}
\tr^{ss}(\Fr_q^j, H(X_K)) - \card{X(\FF_{q^j})}
= q^j \card{\PP^{n-1}(\FF_{q^j})} - q^j \card{Z(\FF_{q^j})}
= - q^j \alpha_Z^j \bm{1}_{2\mid n-1}.
\end{equation*}

Next, we pass from the alternating semi-simple trace $\tr^{ss}(\Fr_q^j, H(X_K))$ to the individual $\tr^{ss}(\Fr_q^j, H^i(X_K))$.
By (2), we have $H^i(X_K) \cong H^i(\PP^n_K)$ for $i\ne n$.
Since $H^i(\PP^n_K)$ is unramified,
we conclude that $\tr^{ss}(\Fr_q^j, H^i(X_K)) = \tr(\Fr_q^j, H^i(\PP^n_k))$ for $i\ne n$, whence
\begin{equation}
\label{EQN:first-concrete-semi-simple-trace-formula}
(-1)^n \tr^{ss}(\Fr_q^j, H^n(X_K))
+ \card{\PP^n(\FF_{q^j})} - q^{nj/2} \bm{1}_{2\mid n}
- \card{X(\FF_{q^j})}
= - q^j \alpha_Z^j \bm{1}_{2\mid n-1}.
\end{equation}

By (2) and the Grothendieck--Lefschetz trace formula, we have
\begin{equation*}
(-1)^n(\card{X(\FF_{q^j})} - \card{\PP^n(\FF_{q^j})})
= \tr(\Fr_q^j, H^n(X_k)) - \tr(\Fr_q^j, H^n(\PP^n_k))
= \sum_{\alpha \in M_0} \alpha^j
- \sum_{\alpha \in M_1} \alpha^j.
\end{equation*}
for some unique pair $(M_0,M_1)$ of \emph{disjoint} multisets $M_0$ and $M_1$.
Eigenvalues from different cohomology groups could cancel,
but certainly each $\alpha\in M_0\cup M_1$ has weight $\le n$.
In any case, plugging the previous display into \eqref{EQN:first-concrete-semi-simple-trace-formula} gives
\begin{equation}
\label{EQN:second-concrete-semi-simple-trace-formula}
\tr^{ss}(\Fr_q^j, H^n(X_K))
= q^{nj/2} \bm{1}_{2\mid n} + (q\alpha_Z)^j \bm{1}_{2\mid n-1}
+ \sum_{\alpha \in M_0} \alpha^j
- \sum_{\alpha \in M_1} \alpha^j.
\end{equation}

It remains to pass from $\tr^{ss}(\Fr_q^j, H^n(X_K))$ to $\tr(\Fr_q^j, H^n(X_K)^I)$.
This we do using the purity of $gr^W_a{H^n(X_K)}$ (assumed in (2)),
following the recipe of \cite{rapoport1990bad}*{p.~268}.
Write $\tr^{ss}(\Fr_q^j, H^n(X_K)) = \sum_{\beta\in EV} \beta^j$, where $EV$ is the multiset of eigenvalues of $\Fr_q$ defined on \cite{rapoport1990bad}*{p.~268}.
Then (by the isomorphism $gr^W_a{H^n} \to gr^W_{-a}{H^n}$ recorded on \cite{rapoport1990bad}*{p.~266}), the multiset $EV$ is symmetric about $q^n$,
i.e.~the multisets $\set{\beta}$ and $\set{q^{2n}/\ol{\beta}}$ coincide.

\emph{Case~1: $M_1 = \emptyset$.}
Then by \eqref{EQN:second-concrete-semi-simple-trace-formula},
the multiset $EV$ consists of $M_0$
plus $q^{n/2}$ (resp.~plus $q\alpha_Z$) if $2\mid n$ (resp.~$2\nmid n$).
In the notation of \cite{rapoport1990bad}*{p.~268},
we thus have $EV_{0\min} = EV$ and $EV_1 = \emptyset$ if $2\mid n$,
and if $2\nmid n$ then $EV_{0\min} = \set{\alpha_Z}$, $EV_1 = EV\setminus \set{\alpha_Z,q\alpha_Z} \belongs M_0$, $EV_{1\min} = EV_1$, and $EV_2 = \emptyset$.
Therefore,
\begin{equation*}
\tr(\Fr_q^j, H^n(X_K)^I)
= \sum_{i\ge 0} \sum_{\beta\in EV_{i\min}} \beta^j
= q^{nj/2} \bm{1}_{2\mid n}
+ \sum_{\alpha \in M_0} \alpha^j
= \tr(\Fr_q^j, H^n(X_k)).
\end{equation*}

\emph{Case~2: $M_1 \ne \emptyset$.}
Then \eqref{EQN:second-concrete-semi-simple-trace-formula} forces
$2\mid n$, $M_1 = \set{q^{n/2}}$, and $EV = M_0$, so
\begin{equation*}
\tr(\Fr_q^j, H^n(X_K)^I)
= \sum_{\alpha\in M_0} \alpha^j
= \tr(\Fr_q^j, H^n(X_k)).
\end{equation*}

In both cases, we are now done,
because (2) already trivially implies
$\tr(\Fr_q^j, H^i(X_K)^I)
= \tr(\Fr_q^j, H^i(X_k))$
for $i\ne n$.
\end{proof}






\section*{Acknowledgements}

I thank
Tim Browning,
Adrian Diaconu,
Jakob Glas,
Dan Petersen,
and Peter Sarnak
for their generous comments,
which encouraged me to properly write up this note.
Adrian Diaconu, Dan Petersen, and Will Sawin made helpful suggestions on Betti numbers and other topics.
This work was supported by the European Union's Horizon~2020 research and innovation program under the Marie Sk\l{}odowska-Curie Grant Agreement No.~101034413.

\bibliographystyle{amsxport}
\bibliography{main.bib}
\end{document}